\documentclass[11pt, parskip=none]{scrartcl}
\usepackage[a4paper,hmargin={2.5cm,2.5cm},vmargin={2.5cm,2.5cm},heightrounded, marginparwidth=2.2cm, marginparsep=0.1cm]{geometry}

\usepackage[square]{natbib} %,numbers
\usepackage{amsmath}
\usepackage{amssymb,amsthm}
\usepackage{stmaryrd}
\usepackage{csquotes}

\usepackage{tikz-cd}
\usetikzlibrary{decorations.markings}
\theoremstyle{plain}
\newtheorem{theorem}{Theorem}[section]
\newtheorem{lemma}[theorem]{Lemma}
\newtheorem{proposition}[theorem]{Proposition}
\newtheorem{corollary}[theorem]{Corollary}

\theoremstyle{definition}
\newtheorem{definition}[theorem]{Definition}
\newtheorem{example}[theorem]{Example}

\theoremstyle{remark}
\newtheorem{remark}[theorem]{Remark}

\counterwithin*{equation}{section}

\RequirePackage[shortlabels]{enumitem}
\newlist{tfae}{enumerate}{1}%
\setlist[tfae,1]{label=(\roman*)}%

\makeatletter
\newcommand{\mylabel}[2]{#2\def\@currentlabel{#2}\label{#1}}
\makeatother

\newcommand{\catfont}[1]{\mathsf{#1}}

\newcommand{\catA}{\catfont{A}}
\newcommand{\catB}{\catfont{B}}
\newcommand{\catC}{\catfont{C}}
\newcommand{\catE}{\catfont{E}}
\newcommand{\catX}{\catfont{X}}

\newcommand{\one}{\textit{\fontfamily{qzc}\selectfont 1}}
\newcommand{\two}{\textit{\fontfamily{qzc}\selectfont 2}}

\newcommand{\Set}{\catfont{Set}}
\newcommand{\Ord}{\catfont{Ord}}
\newcommand{\Met}{\catfont{Met}}
\newcommand{\Cat}{\catfont{Cat}}
\newcommand{\Dist}{\catfont{Dist}}
\newcommand{\CAT}{\catfont{CAT}}
\newcommand{\Cats}[1]{#1\text{-}\catfont{Cat}}
\newcommand{\CATs}[1]{#1\text{-}\catfont{CAT}}
\newcommand{\Dists}[1]{#1\text{-}\catfont{Dist}}

\newcommand{\DSets}[1]{#1\text{-}\catfont{DSet}}
\newcommand{\Setss}[1]{\catfont{Set}{/\!\!/}#1}
\newcommand{\nSetss}[1]{\catfont{Set}{|\!|}#1}
\newcommand{\Catss}[1]{\catfont{Cat}{/\!\!/}#1}
\newcommand{\CATss}[1]{\catfont{CAT}{/\!\!/}#1}
\newcommand{\Distss}[1]{\catfont{Dist}{/\!\!/}#1}
\newcommand{\SetV}{\Setss{\mcatV}}
\newcommand{\DistV}{\Distss{\mcatV}}
\newcommand{\nSetV}{\nSetss{\mcatV}}
\newcommand{\CatV}{\Catss{\mcatV}}
\newcommand{\CATV}{\CATss{\mcatV}}

\newcommand{\ncatfont}[1]{\mathbb{#1}}
\newcommand{\ncatA}{\ncatfont{A}}
\newcommand{\ncatB}{\ncatfont{B}}
\newcommand{\ncatN}{\ncatfont{N}}
\newcommand{\ncatX}{\ncatfont{X}}
\newcommand{\ncatY}{\ncatfont{Y}}

\newcommand{\ncatE}{\ncatfont{E}}

\newcommand{\VLipDot}{\mcatV_{\otimes}\text{-}\catfont{Lip}_{\odot}}
\newcommand{\VLip}{\mcatV\text{-}\catfont{Lip}}

\newcommand{\moncatfont}[1]{\mathcal{#1}}
\newcommand{\mcatW}{\moncatfont{W}}
\newcommand{\mcatV}{\moncatfont{V}}
\newcommand{\mcatU}{\moncatfont{U}}

\DeclareMathOperator{\ob}{ob}
\DeclareMathOperator{\ev}{ev}
\DeclareMathOperator{\comp}{comp}
\DeclareMathOperator{\unit}{unit}
\DeclareMathOperator{\assoc}{assoc}
\DeclareMathOperator{\sym}{sym}

\newcommand{\op}{\mathrm{op}}

\DeclareMathOperator{\mO}{O}
\DeclareMathOperator{\ms}{s}
\DeclareMathOperator{\mi}{i}

\newcommand{\cb}[1]{{#1}^{\mathrm{fun}}}
\newcommand{\db}[1]{{#1}^{\mathrm{dist}}}

\newcommand{\kk}{\mathrm{k}}

\newcommand*{\quot}[2]{{^{\textstyle #1}\big/_{\textstyle #2}}}

\RequirePackage{mathtools}

\makeatletter

\def\slashedarrowfill@#1#2#3#4#5{%
  $\m@th\thickmuskip0mu\medmuskip\thickmuskip\thinmuskip\thickmuskip
   \relax#5#1\mkern-7mu%
   \cleaders\hbox{$#5\mkern-2mu#2\mkern-2mu$}\hfill
   \mathclap{#3}\mathclap{#2}%
   \cleaders\hbox{$#5\mkern-2mu#2\mkern-2mu$}\hfill
   \mkern-7mu#4$%
}

\newcommand*{\rightrelarrowfill@}{\slashedarrowfill@\relbar\relbar{\raisebox{0pc}{$\mapstochar$}}\rightarrow}
\newcommand*{\xrelto}[2][]{\ext@arrow 0055{\rightrelarrowfill@}{\;#1\;}{\;#2\;}}

\newcommand*{\rightmodarrowfill@}{\slashedarrowfill@\relbar\relbar{\raisebox{0pc}{$\hspace{1pt}\circ$}}\rightarrow}
\newcommand*{\xmodto}[2][]{\ext@arrow 0055{\rightmodarrowfill@}{\;#1\;}{\;#2\;}}

\newcommand*{\modto}{\xmodto{\;}}

\makeatother

\tikzset{
  relational/.style={
    outer sep=3pt,
    decoration={
      markings,
      mark=at position 0.5 with {\node[transform shape] (tempnode) {\tiny $\rvert$};},
    },
    postaction={decorate},
  },
}

\tikzset{
  distrib/.style={
    outer sep=3pt,
    decoration={
      markings,
      mark=at position 0.5 with {\node[transform shape] (tempnode) {\tiny
          \rmfamily o};},
    },
    postaction={decorate},
  },
}

\tikzset{%
  symbol/.style={%
    draw=none, every to/.append style={%
      edge node={node [sloped, allow upside down,auto=false]{$#1$}}} } }

\usepackage[hypertexnames=false]{hyperref}

\hypersetup{
  colorlinks = true,
  citecolor= [rgb]{0,0.75,0}, % dark green -- default: green
  urlcolor=[rgb]{0.4,0,0.4}, % dark violet -- default: magenta
  linkcolor=[rgb]{0.8,0,0} % dark red -- default: red
}

\usepackage[auth-lg]{authblk}

\title{Notions of Cauchy completeness for normed categories}

\author[1]{Dirk Hofmann\thanks{This work is supported by National Funds through the Portuguese funding agency, FCT -- Fundação para a Ciência e a Tecnologia, within project UID/4106/2025 (CIDMA -- Center for Research and Development in Mathematics and Applications).}}

\author[2]{Walter Tholen}

\affil[1]{University of Aveiro, CIDMA, Department of Mathematics, Portugal,\texttt{dirk@ua.pt}}

\affil[2]{York University, Toronto, Canada, \texttt{tholen@yorku.ca}}

\date{\today}

\dedication{Dedicated to Robert Paré}

\begin{document}

\maketitle

\begin{abstract}
  As already mentioned by Lawvere in his 1973 paper, the characterisation of Cauchy completeness of metric spaces in terms of representability of adjoint distributors amounts to the idempotent-split property of an ordinary category when the governing symmetric monoidal-closed category is changed from the extended real half-line to the category of sets. In this paper, for any commutative quantale \(\mcatV\), we extend these two characterisations of Lawvere-style completeness to \(\mcatV\)-normed categories, thus replacing \([0,\infty]\) and \(\Set\) more generally by the category \(\SetV\) of \(\mcatV\)-normed sets. We also establish improvements of recent results regarding the normed convergence of Cauchy sequences in two important \(\mcatV\)-normed categories.
\end{abstract}

\emph{Mathematics Subject Classification}: 18D20, 18D60, 18B35, 54A20.

\emph{Keywords}: Cauchy-complete metric space, normed category, normed distributor, quantale, Lawvere-complete normed category.

\section*{Introduction}

The primary motivation for this paper stems from the following two achievements of Lawvere's famous article \citep{Law73}, which is widely known for the interpretation of (generalized) metric spaces as small categories enriched in the real half-line  \([0,\infty]\):
\begin{itemize}
\item The characterisation of Cauchy completeness of metric spaces as an instance of a categorical notion of completeness (in terms of the representabilty of adjoint distributors) for categories enriched in a symmetric monoidal-closed category \(\mcatW\) which, for \(\mcatW=\Set\), amounts to the idempotent-split property of ordinary categories;
\item The suggestion of a notion of \([0,\infty]\)-normed category as a category enriched in the symmetric monoidal-closed category \(\mcatW=\Setss{[0,\infty]}\) of \([0,\infty]\)-normed sets, and the close interaction of such normed categories with metric spaces.
\end{itemize}
The juxtaposition of these two achievements naturally leads to the question of \emph{how to characterize \([0,\infty]\)-normed categories that are complete in Lawvere's sense}\,? The principal result of this paper (Theorem \ref{d:thm:2}) anwers this question, even in the more general setting employed in the recent article \citep{CHT25} where \([0,\infty]\) has been replaced by a commutative \emph{quantale} \(\mcatV\), \emph{i.e.}, by any small and ``thin'' symmetric monoidal-closed category.

In order to characterise \emph{Lawvere completeness} for a \(\mcatV\)-normed category \(\ncatA\) (that is: for a \(\SetV\)-enriched category \(\ncatA\)), one needs to carefully analyse the notion of adjoint \(\SetV\)-distributor (or, in Lawvere's terminology, of adjoint \(\SetV\)-bimodule). Since a \(\mcatV\)-normed category \(\ncatA\) may be seen as an ordinary category carrying additional structure, it seemed natural to us trying to ``lift'' the proof of the characterisation of adjoint distributors as given by Borceux in \citep{Bor94} from the \(\Set\)- to the \(\SetV\)-level. But when encountering some intricacies in the process we found it necessary to clarify, and be very precise about, various elements of enriched category theory, some of which don't seem to be easily available in sufficient detail in the literature, in particular those concerning enriched distributors. We therefore decided to summarise these elements comprehensively in the introductory Section~\ref{sec:background} and to illustrate them in the three cases of interest for this paper, namely when the enriching category \(\mcatW\) is the category \(\Set\), a quantale \(\mcatV\), or the category \(\SetV\) of \(\mcatV\)-normed sets.

With all needed elements of enriched category readily available, in the principal Section~\ref{sec:compl-la-lawv} of the paper we are then prepared to analyse Lawvere's completeness notion for \(\mcatW\)-categories in the three cases of interest. In presenting the largely known characterisations in the cases  \(\mcatW=\mcatV\) and \(\mcatW=\Set\), following \citep{AL21} we found that the use of the \emph{Isbell conjugation} helps shortening the proofs, and this applies also to the case of principal interest, \(\mcatW=\SetV\). It may not be surprising that, for a \(\mcatV\)-normed category \(\ncatA\) to be Lawvere complete, it is necessary that the ordinary category \(\ncatA_{\circ}\) of \(\ncatA\)-morphism morphisms with norm at least \(\kk\) (the tensor neutral element in \(\mcatV\)) be so. Since a left adjoint distributor \(\catE\modto\ncatA_{\circ}\) (with \(\catE\) the one-morphism category) must therefore be a retract of a representable functor (seen as a distributor), it seems natural to expect that the needed additional ingredient for Lawvere completeness of \(\ncatA\) as a \(\mcatV\)-normed category is the condition that the natural transformations involved in the presentation as a retract should comply in some way with the norm of \(\ncatA\). The specification of this compliance is described in Definition~\ref{w:def:1} and then leads to Theorem~\ref{d:thm:2} which clarifies our initial vague expectation.

Unlike the self-dual notion of Lawvere completeness (see Remark~\ref{w:rem:1}), the notion of \emph{normed colimit} of a sequence in a \(\mcatV\)-normed category as introduced in \citep{CHT25} (in generalization of the notion of \emph{forward limit} of a generalized metric space as in \citep{BBR98}) is clearly not invariant under dualization. In Section~\ref{sec:compl-la-cauchy} we briefly revisit the resulting notion of \emph{Cauchy cocompleteness} of a \(\mcatV\)-normed category in order to improve upon on a couple of results of \citep{CHT25}. Expanding further on the note \citep{CHT25a}, but without imposing any additional condition on \(\mcatV\), we show that the \(\mcatV\)-normed category \(\nSetV\) of \(\mcatV\)-normed sets and arbitrary maps as morphisms is Cauchy cocomplete (Theorem~\ref{d:thm:3}). And under a modest restriction on the quantale that already appeared in \citep{CHT25} we provide an appropriate \(\mcatV\)-generalization of the fact that the normed category of generalized metric spaces with arbitrary maps is Cauchy cocomplete as well (Theorem~\ref{d:thm:4}).

Finally we note that the notion of Lawvere completeness is intrinsically double-categorical as it relates adjoint ``loose'' arrows (distributors) and ``strict'' arrows (functors). A natural extension of the notion of Lawvere completeness to double categories is introduced in \citep{Par21} and further studied in \citep{Nie25}.

\section{Background from enriched category theory}
\label{sec:background}

\subsection{Symmetric monoidal categories}
\label{sec:symm-mono-categ}

A monoidal category (see \citep{EK66, Kel82, Bor94a, Mac98}) is a category \(\mcatW\) equipped with a bifunctor (called the tensor product)
\begin{displaymath}
  \otimes \colon\mcatW\times\mcatW \longrightarrow\mcatW
\end{displaymath}
and an object \(E\) (called the unit) together with natural isomorphisms
\begin{displaymath}
  \assoc \colon (A\otimes B)\otimes C \longrightarrow
  A\otimes (B\otimes C)
\end{displaymath}
and
\begin{displaymath}
  \text{left unit}\colon E\otimes A \longrightarrow A
  \quad\text{and}\quad
  \text{right unit}\colon A\otimes E \longrightarrow A
\end{displaymath}
satisfying the following coherence conditions: for all objects \(A\), \(B\), \(C\) and \(D\) of \(\mcatW\), the diagrams
\begin{displaymath}
  \begin{tikzcd}
    ((A\otimes B)\otimes C)\otimes D %
    \ar{r}{\assoc} %
    \ar{d}[swap]{\assoc\otimes 1} %
    & (A\otimes B)\otimes (C\otimes D) %
    \ar{r}{\assoc} %
    & A\otimes (B\otimes (C\otimes D))\\
    (A\otimes (B\otimes C))\otimes D %
    \ar{rr}[swap]{\assoc} %
    && A\otimes ((B\otimes C)\otimes D) %
    \ar{u}[swap]{1\otimes\assoc}
  \end{tikzcd}
\end{displaymath}
and
\begin{displaymath}
  \begin{tikzcd}
    (A\otimes E)\otimes B %
    \ar{rr}{\assoc} %
    \ar{dr}[swap]{(\text{right unit})\otimes 1}
    && A\otimes (E\otimes B) %
    \ar{dl}{1\otimes(\text{left unit})}\\
    & A\otimes B
  \end{tikzcd}
\end{displaymath}
are commutative. A monoidal category \(\mcatW\) is \emph{symmetric} whenever there are also natural isomorphisms
\begin{displaymath}
  \text{sym}\colon A\otimes B \longrightarrow B\otimes A
\end{displaymath}
given so that the diagrams
\begin{displaymath}
  \begin{tikzcd}[column sep=large]
    (A\otimes B)\times C %
    \ar{r}{\sym\otimes 1} %
    \ar{d}[swap]{\assoc} %
    & (B\otimes A)\otimes C %
    \ar{d}{\assoc}\\
    A\otimes(B\otimes C) %
    \ar{d}[swap]{\sym} %
    & B\otimes (A\otimes C) %
    \ar{d}{1\otimes\sym}\\
    (B\otimes C)\otimes A %
    \ar{r}[swap]{\assoc} %
    & B\otimes(C\otimes A),
  \end{tikzcd}
\end{displaymath}
\begin{displaymath}
  \begin{tikzcd}
    A\otimes E %
    \ar{r}{\sym} %
    \ar{dr}[swap]{\text{right unit}} %
    & E\otimes A %
    \ar{d}{\text{left unit}}\\
    & A,
  \end{tikzcd}
  \qquad
  \begin{tikzcd}
    A\otimes B %
    \ar{r}{\sym} %
    \ar{rd}[swap]{1} %
    & B\otimes A %
    \ar{d}{\sym}\\
    & A\otimes B
  \end{tikzcd}
\end{displaymath}
are commutative, for all objects \(A\), \(B\) and \(C\) of \(\mcatW\). A symmetric monoidal category \(\mcatW\) is called \emph{closed} whenever, for every object \(A\) of \(\mcatW\), the functor \(-\otimes A \colon\mcatW\to\mcatW\) has a right adjoint.

Below we list our principal examples of closed symmetric monoidal categories.
\begin{example}
  \label{d:ex:1}
  \begin{enumerate}
  \item The category \(\mcatW=\Set\) of sets and functions is symmetric monoidal closed, with the usual Cartesian structure.
  \item A complete lattice, viewed as a category, is symmetric monoidal closed if and only if it is a commutative and unital quantale (see \citep{Ros90a}), in this paper simply designated as \emph{quantale}. Whenever we consider a quantale, we typically write \(\mcatV\) instead of \(\mcatW\) and denote the unit element of \(\mcatV\) by \(\kk\).
  \item For every quantale \(\mcatV\), the category \(\mcatW=\SetV\) of \(\mcatV\)-normed sets and \(\mcatV\)-normed maps is symmetric monoidal closed. Here a \emph{$\mcatV$-normed} set is a set $A$ that comes with a function $|{-}|\colon A \to\mcatV$, and a \emph{$\mcatV$-normed map} $(A,|{-}|)\to(B,|{-}|)$ is a mapping $f \colon A \to B$ satisfying $|a|\leq|fa|$ for all $a\in A$.
    \begin{displaymath}
      \begin{tikzcd}
        A %
        \ar{rr}{f}[swap, inner sep=1em]{\leq} %
        \ar{dr}[swap]{|{-}|} %
        && B %
        \ar{ld}{|{-}|} \\
        & \mcatV %
      \end{tikzcd}
    \end{displaymath}
    For $\mcatV$-normed sets $A$ and $B$, their tensor product $A\otimes B$ is carried by the Cartesian product $A\times B$ of the sets \(A\) and \(B\), normed by $|(a,b)|=|a|\otimes|b|$ in $\mcatV$. The unit \(\mcatV\)-normed set $E$ is the set $\{\star\}$ normed by $|{\star}|=\kk$. The right adjoint of \(-\otimes A\) is described in Example~\ref{d:ex:3}, and for more details we refer to \citep{Law73, BG75, CHT25}.
  \end{enumerate}
\end{example}

\subsection{Enriched categories, functors, and natural transformations}
\label{sec:enrich-categ-funct}

Let \(\mcatW\) be a symmetric monoidal category. A \emph{\(\mcatW\)-category} \(\mathcal{A}\) consists of a collection of objects, denoted by \(\ob \mathcal{A}\), together with a family of hom-objects \(\mathcal{A}(A,B)\) in \(\mcatW\), indexed by pairs \((A,B)\) of objects of \(\mathcal{A}\), composition arrows
\begin{displaymath}
  \mathcal{A}(B,C)\otimes \mathcal{A}(A,B)
  \xrightarrow{\quad\comp\quad} \mathcal{A}(A,C)
\end{displaymath}
and identity arrows
\begin{displaymath}
  E \xrightarrow{\quad\unit\quad} \mathcal{A}(A,A),
\end{displaymath}
subject to the associativity and unit axioms: for all objects \(A\), \(B\), \(C\) and \(D\) of \(\mathcal{A}\), the diagrams
\begin{displaymath}
  \begin{tikzcd}
    (\mathcal{A}(C,D)\otimes \mathcal{A}(B,C))\otimes \mathcal{A}(A,B) %
    \ar{rr}{\assoc} %
    \ar{d}[swap]{\comp\otimes 1} %
    && \mathcal{A}(C,D)\otimes (\mathcal{A}(B,C)\otimes \mathcal{A}(A,B)) %
    \ar{d}{1\otimes\comp}\\
    \mathcal{A}(B,D)\otimes \mathcal{A}(A,B) %
    \ar{rd}[swap]{\comp} %
    && \mathcal{A}(C,D)\otimes \mathcal{A}(A,C) %
    \ar{dl}{\comp}\\
    & \mathcal{A}(A,D)
  \end{tikzcd}
\end{displaymath}
and
\begin{displaymath}
  \begin{tikzcd}[column sep=large]
    \mathcal{A}(B,B)\otimes \mathcal{A}(A,B) %
    \ar{r}{\comp} %
    & \mathcal{A}(A,B) %
    & \mathcal{A}(A,B)\otimes \mathcal{A}(A,A) %
    \ar{l}[swap]{\comp} \\
    E\otimes \mathcal{A}(A,B) %
    \ar{u}{\unit\otimes 1} %
    \ar{ur}[swap]{\text{left unit}} %
    && \mathcal{A}(A,B)\otimes E %
    \ar{u}[swap]{1\otimes\unit} %
    \ar{ul}{\text{right unit}}
  \end{tikzcd}
\end{displaymath}
are commutative.

Given \(\mcatW\)-categories \(\mathcal{A}\) and \(\mathcal{B}\), a \emph{\(\mcatW\)-functor} \(F \colon \mathcal{A}\to \mathcal{B}\) consists of a function \(F\) which associates to each object \(A\) of \(\mathcal{A}\) an object \(FA\) of \(\mathcal{B}\), and of arrows
\begin{displaymath}
  F=F_{A,B}\colon\mathcal{A}(A,B)\longrightarrow \mathcal{B}(FA,FB)
\end{displaymath}
in \(\mcatW\) compatible with the identities and composition in the sense that the diagrams
\begin{displaymath}
  \begin{tikzcd}
    \mathcal{A}(B,C)\otimes \mathcal{A}(A,B) %
    \ar{r}{\comp} %
    \ar{d}[swap]{F\otimes F} %
    & \mathcal{A}(A,C) %
    \ar{d}{F} \\
    \mathcal{B}(FB,FC)\otimes \mathcal{B}(FA,FB) %
    \ar{r}[swap]{\comp} %
    & \mathcal{B}(FA,FC),
  \end{tikzcd}
  \quad
  \begin{tikzcd}
    & E
    \ar{dl}[swap]{\unit}
    \ar{dr}{\unit}\\
    \mathcal{A}(A,A)
    \ar{rr}[swap]{F}
    && \mathcal{B}(FA,FA)
  \end{tikzcd}
\end{displaymath}
are commutative, for all objects \(A\), \(B\) and \(C\) of \(\mathcal{A}\).

Finally, for \(\mcatW\)-functors \(F,G \colon \mathcal{A}\to \mathcal{B}\), a \emph{\(\mcatW\)-natural transformation} \(\alpha \colon F\to G\) is a family \(\alpha=(\alpha_A)_{A\in\ob\mathcal{A}}\) of arrows \(\alpha_A \colon E\to \mathcal{B}(FA,GA)\) in \(\mcatW\) subject to the following \(\mcatW\)-naturality condition: for all objects \(A,B\) in \(\mathcal{A}\), the diagram
\begin{displaymath}
  \begin{tikzcd}[column sep=large]
    & E\otimes \mathcal{A}(A,B) %
    \ar{r}{\alpha_B\otimes F} %
    & \mathcal{B}(FB,GB)\otimes \mathcal{B}(FA,FB)
    \ar{dr}\\
    \mathcal{A}(A,B)\ar{ur}\ar{dr} %
    &&& \mathcal{B}(FA,GB)\\
    & \mathcal{A}(A,B)\otimes E %
    \ar{r}[swap]{G\otimes\alpha_A} %
    & \mathcal{B}(GA,GB)\otimes \mathcal{B}(FA,GA) %
    \ar{ur}
  \end{tikzcd}
\end{displaymath}
commutes. For \(\mcatW\)-functors \(K \colon \mathcal{B}\to \mathcal{C}\) and \(L \colon \mathcal{D}\to \mathcal{A}\), the \emph{horizontal composites} \(K\alpha \colon KF\to KG\) and \(\alpha L \colon FL\to GL\) have components
\begin{displaymath}
  \begin{tikzcd}
    E\ar{r}[swap]{\alpha_A}\ar[bend left=25]{rr}{(K\alpha)_A} %
    & \mathcal{B}(FA,GA)\ar{r}[swap]{K} %
    & \mathcal{C}(KFA,KGA)
  \end{tikzcd}
\end{displaymath}
and
\begin{displaymath}
  E\xrightarrow{\quad(\alpha L)_D=\alpha_{LD}\quad}
  \mathcal{C}(FLD,GLD),
\end{displaymath}
respectively. Given also a \(\mcatW\)-functor \(H \colon \mathcal{A}\to \mathcal{B}\) and a \(\mcatW\)-natural transformation \(\beta \colon G\to H\) with components \(\beta_A \colon E\to \mathcal{B}(GA,HA)\), the \emph{vertical composite} \(\beta\alpha\) has components \((\beta\alpha)_A\) given by
\begin{displaymath}
  E\cong E\otimes E\xrightarrow{\quad\beta_A\otimes\alpha_A\quad}
  \mathcal{B}(GA,HA)\otimes \mathcal{B}(FA,GA)
  \xrightarrow{\quad\comp\quad} \mathcal{B}(FA,HA).
\end{displaymath}
Regarding the vertical composition, the \(\mcatW\)-natural transformation \(1_F=(1_{FA})_{A\in\ob \mathcal{A}}\) acts as an identity, moreover, we have the \emph{interchange laws}
\begin{align*}
  K1_F&=1_{KF}, & K(\beta\alpha) &=(K\beta)(K\alpha),\\
  1_FL&=1_{FL}, & (\beta\alpha)L &=(\beta L)(\alpha L).
\end{align*}
We obtain the 2-category \(\CATs{\mcatW}\) of \(\mcatW\)-categories, \(\mcatW\)-functors and \(\mcatW\)-natural transformations, and its full subcategoy \(\Cats{\mcatW}\) of small \(\mcatW\)-categories.

We also recall that the monoidal structure of \(\mcatW\) induces naturally a monoidal structure on \(\CATs{\mcatW}\) and \(\Cats{\mcatW}\): for \(\mcatW\)-categories \(\mathcal{A}\) and \(\mathcal{B}\), their tensor product \(\mathcal{A}\otimes \mathcal{B}\) has \(\ob\mathcal{A}\times\ob\mathcal{B}\) as object-class, and
\begin{displaymath}
  \mathcal{A}\otimes \mathcal{B}((A,B),(A',B'))
  = \mathcal{A}(A,A')\otimes \mathcal{B}(B,B').
\end{displaymath}
The neutral element is given by the \(\mcatW\)-category \(\mathcal{E}\) with one object (say, \(\ob\mathcal{E}=\{\star\}\)) and with \(\mathcal{E}(\star,\star)=E\).

\begin{example}
  For \(\mcatW=\Set\), we have \(\Cats{\Set}=\Cat\) and \(\CATs{\Set}=\CAT\), the usually 2-category of small and of locally small categories respectively, functors and natural transformations, with its Cartesian structure. We usually denote ordinary (small) categories by \(\catA\), \(\catB\), \dots.
\end{example}

\begin{example}\label{d:ex:4}
  We consider now the case that \(\mcatW=\mcatV\) is a quantale. Since \(\mcatV\) is a thin category, all coherence conditions are automatically satisfied, moreover, all structure arrows become properties. Therefore a \emph{\(\mcatV\)-category} \(X\) consists of a set \(X\) and a function \(X(-,-)\colon X\times X\to\mcatV\) satisfying
  \begin{displaymath}
    \kk\leq X(x,x)
    \quad\text{and}\quad
    X(y,z)\otimes X(x,y)\leq X(x,z).
  \end{displaymath}
  A \emph{\(\mcatV\)-functor} \(f \colon X\to Y\) must satisfy \(X(x,x')\leq Y(fx,fx')\), for all \(x,x'\in X\). Finally, for \(\mcatV\)-functors \(f,g \colon X\to Y\), there is a (unique) \(\mcatV\)-natural transformation \(f\to g\) if and only if, for all \(x\in X\), \(\kk\leq Y(fx,gx)\), and we write \(f\leq g\) in that case \citep{HST14}.

  For \(\mcatV=\one\) the terminal quantale, \(\Cats{\one}\) is simply the category \(\Set\) of sets and functions, it becomes a 2-category with the discrete structure on hom-sets. For \(\mcatV=\two\) the two-element chain \(\two=\{0\leq 1\}\) with \(\otimes=\wedge\) and \(\kk=1\), the 2-category \(\Cats{\two}\) is the 2-category of ordered sets, monotone maps and the point-wise order on hom-sets; here an order relation on a set \(X\) is a reflexive and transitive binary relation on \(X\). Finally, for \(\mcatV=[0,\infty]\) ordered by the natural greater-or-equal relation \(\geq\) and with \(\otimes=+\) and \(\kk=0\), the 2-category \(\Cats{[0,\infty]}\) is the 2-category \(\Met\) of (generalised) metric spaces, non-expansive maps and the point-wise order on hom-sets (see \citep{Law73}).
\end{example}

\begin{example}
  Let now \(\mcatV\) be a quantale and consider \(\mcatW=\SetV\). A \(\SetV\)-category \(\ncatA\) is an ordinary category where, moreover, \(\ncatA(a,b)\) is an object of \(\SetV\) for all objects \(a,b\) in \(\ncatA\). The identity arrows \(E\to\ncatA(a,a)\) are in \(\SetV\), that is, \(\kk\leq |1_a|\), and the composition arrows
  \begin{displaymath}
    \ncatA(b,c)\otimes\ncatA(a,b) \longrightarrow\ncatA(a,b)
  \end{displaymath}
  are in \(\SetV\), that is,
  \begin{displaymath}
    |g|\otimes |f|\leq |g\cdot f|.
  \end{displaymath}
  A \(\SetV\)-functor \(F \colon\ncatA\to\ncatB\) is an ordinary functor with arrows
  \begin{displaymath}
    F \colon\ncatA(a,a') \longrightarrow\ncatB(Fa,Fa')
  \end{displaymath}
  in \(\SetV\), that is \(|f|\leq|Ff|\) for every \(f \colon a\to a'\) in \(\ncatA\). For  \(\SetV\)-functors \(T,S \colon\ncatA\to\ncatB\), a \(\SetV\)-natural transformation \(\alpha \colon T\to S\) is an ordinary natural transformation such that the family \((\alpha_a \colon E\to\ncatB(Ta,Sa))_a\) indexed by objects of \(\ncatA\) lives in \(\SetV\), that is, \(\alpha_a \colon Ta\to Sa\) must satisfy \(\kk\leq|\alpha_a|\) for all objects \(a\) in \(\ncatA\). We refer to \(\SetV\)-categories, \(\SetV\)-functors and \(\SetV\)-natural transformations as \emph{\(\mcatV\)-normed categories}, \emph{\(\mcatV\)-normed functors} and \emph{\(\mcatV\)-normed natural transformations}, respectively; and simply write \(\CatV\) and \(\CATV\) instead of \(\Cats{(\SetV)}\) and \(\CATs{(\SetV)}\), respectively. The \emph{tensor product} $\ncatA\otimes\ncatB$ of the \(\mcatV\)-normed categories \(\ncatA\) and \(\ncatB\) is carried by the ordinary category $\ncatA\times\ncatB$, structured by
  \begin{displaymath}
    |(f,g)|=|f|\otimes|g|;
  \end{displaymath}
  and with neutral object \(\ncatE\) the one-arrow category with \(|1_{\star}|=\kk\). For more details see \citep{CHT25}.

  In particular, for \(\mcatV=\one\), the 2-category \(\Catss{\one}\) is equivalent to the 2-category \(\Cat\). A \(\two\)-normed category can be equivalently described as category \(\ncatA\) equipped with a wide subcategory, a \(\two\)-normed functor is a functor which restricts to the chosen subcategories, and a \(\two\)-normed natural transformation has all components in the chosen subcategory.
\end{example}

\subsection{Change of base}
\label{sec:change-base}

Given symmetric monoidal categories \(\mcatU\) and \(\mcatW\), a \emph{lax monoidal functor}
\begin{displaymath}
  \Gamma \colon\mcatW \longrightarrow\mcatU
\end{displaymath}
is a functor that comes with natural \(\mcatU\)-morphisms
\begin{equation}
  \label{d:eq:7}
  E_{\mathcal{U}}\xrightarrow{\quad\unit\quad}\Gamma(E_{\mcatW})
  \quad\text{and}\quad
  \Gamma(A)\otimes\Gamma(B)
  \xrightarrow{\quad\tau=\tau_{A,B}\quad}\Gamma(A\otimes B)
\end{equation}
satisfying the coherence equations (see \citep{Mac98}, for instance)
\begin{displaymath}
  \begin{tikzcd}
    \Gamma(A)\otimes(\Gamma(B)\otimes \Gamma(C)) %
    \ar{r}{\mathrm{assoc}} %
    \ar{d}[swap]{1\otimes\tau} %
    & (\Gamma(A)\otimes\Gamma(B))\otimes\Gamma(C) %
    \ar{d}{\tau\otimes 1}\\
    \Gamma(A)\otimes \Gamma(B\otimes C) %
    \ar{d}[swap]{\tau} %
    & \Gamma(A\otimes B)\otimes\Gamma C %
    \ar{d}{\tau}\\
    \Gamma(A\otimes(B\otimes C)) %
    \ar{r}[swap]{\Gamma(\mathrm{assoc})} %
    & \Gamma((A\otimes B)\otimes C),
  \end{tikzcd}
\end{displaymath}
\begin{displaymath}
  \begin{tikzcd}
    \Gamma A\otimes E_{\mcatU} %
    \ar{r}{\text{right unit}} %
    \ar{d}[swap]{1\otimes\unit}
    & \Gamma A\\
    \Gamma(A)\otimes\Gamma E_{\mcatW} %
    \ar{r}[swap]{\tau} %
    & \Gamma(A\otimes E_{\mcatW}), %
    \ar{u}[swap]{\Gamma(\text{right unit})}
  \end{tikzcd}
  \qquad
  \begin{tikzcd}
    E_{\mcatU}\otimes\Gamma A %
    \ar{r}{\text{left unit}} %
    \ar{d}[swap]{\unit\otimes 1}
    & \Gamma A\\
    \Gamma(E_{\mcatW})\otimes \Gamma(A) %
    \ar{r}[swap]{\tau} %
    & \Gamma(E_{\mcatW}\otimes A), %
    \ar{u}[swap]{\Gamma(\text{left unit})}
  \end{tikzcd}
\end{displaymath}
for all objects \(A\), \(B\) and \(C\) of \(\mcatW\). Furthermore, \(\Gamma\) is called \emph{strong monoidal} whenever the arrows \eqref{d:eq:7} are \(\mcatU\)-isomorphisms, \(\Gamma\) is called \emph{strict monoidal} whenever these arrows are \(\mcatU\)-identities, and \(\Gamma\) is called \emph{symmetric} whenever, for all objects \(A\) and \(B\) of \(\mcatW\), the diagram
\begin{displaymath}
  \begin{tikzcd}
    \Gamma(A)\otimes \Gamma(B) %
    \ar{r}{\text{sym}} %
    \ar{d}[swap]{\tau}
    & \Gamma(B)\otimes \Gamma(A) %
    \ar{d}{\tau}\\
    \Gamma(A\otimes B) %
    \ar{r}[swap]{\Gamma(\text{sym})} %
    & \Gamma(B\otimes A)
  \end{tikzcd}
\end{displaymath}
is commutative.

Each lax monoidal functor \(\Gamma\colon\mcatW \to\mcatU\) induces 2-functors
\begin{displaymath}
  \cb{\Gamma} \colon\Cats{\mcatW}\longrightarrow\Cats{\mcatU}
  \quad\text{and}\quad
  \cb{\Gamma} \colon\CATs{\mcatW}\longrightarrow\CATs{\mcatU}
\end{displaymath}
sending
\begin{itemize}
\item the \(\mcatW\)-category \(\mathcal{A}\) to the \(\mcatU\)-category \(\cb{\Gamma}\mathcal{A}\) with the same collection of objects and with
  \begin{displaymath}
    (\cb{\Gamma}\mathcal{A})(A,B)=\Gamma(\mathcal{A}(A,B)),
  \end{displaymath}
\item the \(\mcatW\)-functor \(F \colon \mathcal{A}\to \mathcal{B}\) to the \(\mathcal{U}\)-functor \(\cb{\Gamma}F \colon \cb{\Gamma}\mathcal{A}\to\cb{\Gamma}\mathcal{B}\) with the same effect as \(F\) on objects and with \((\cb{\Gamma}F)_{A,A'}=\Gamma(F_{A,A'})\),
  \begin{displaymath}
    \cb{\Gamma}\mathcal{A}(A,A')=\Gamma(\mathcal{A}(A,A'))
    \xrightarrow{\quad\Gamma(F_{A,A'})\quad}
    \Gamma(\mathcal{B}(FA,FA'))
    =\cb{\Gamma}\mathcal{B}(\cb{\Gamma}F(A),\cb{\Gamma}F(A'))
  \end{displaymath}
\item the \(\mcatW\)-natural transformation \(\alpha \colon F\to G\) with \(\mcatW\)-functors \(F,G \colon \mathcal{A}\to \mathcal{B}\) and \(\alpha=(\alpha_A \colon E_{\mcatW}\to \mathcal{B}(FA,GA))_{A\in\ob \mathcal{A}}\) to the \(\mcatU\)-natural transformation \(\cb{\Gamma}\alpha\colon \cb{\Gamma}F\to\cb{\Gamma}G\) with components
  \begin{displaymath}
    E_{\mcatU}\xrightarrow{\quad\unit\quad}\Gamma(E_{\mcatW})
    \xrightarrow{\quad\Gamma(\alpha_A)\quad}
    \Gamma(\mathcal{B}(FA,GA))
    =\cb{\Gamma}\mathcal{B}(\cb{\Gamma}F(A),\cb{\Gamma}G(A)).
  \end{displaymath}
\end{itemize}
We note that \(\cb{\Gamma}(\mathcal{A}^{\op})=\cb{\Gamma}(\mathcal{A})^{\op}\). Moreover, \(\cb{\Gamma}\) is also (strong, strict, symmetric) monoidal if \(\Gamma \colon\mcatW\to\mcatU\) is so since then the arrows~\eqref{d:eq:7} induce \(\mcatU\)-functors (\(\mcatU\)-isomorphisms, \(\mcatU\)-identities)
\begin{displaymath}
  \mathcal{E}_{\mcatU}\longrightarrow\cb{\Gamma}(\mathcal{E}_{\mcatW})
  \quad\text{and}\quad
  \cb{\Gamma}(\mathcal{A})\otimes\cb{\Gamma}(\mathcal{B})
  \longrightarrow\cb{\Gamma}(\mathcal{A}\otimes \mathcal{B})
\end{displaymath}
satisfying the corresponding coherence conditions.

\begin{proposition}
  \label{d:prop:2}
  For each quantale \(\mcatV\), the forgetful functor
  \begin{displaymath}
    \mO \colon\SetV \longrightarrow\Set
  \end{displaymath}
  is symmetric strict monoidal and topological.
\end{proposition}
For the notion of topological functor we refer to \citep{HST14}, for instance. We recall from \citep{CHT25} that, for a family \((f_i \colon A\to B_i)_{i\in I}\) of maps where each \(B_i\) is normed, the \emph{initial structure} on \(A\) is given by
\begin{displaymath}
  |a|=\bigwedge_{i\in I}|f_i(a)|;
\end{displaymath}
likewise, for a family \((g_i \colon A_i\to B)_{i\in I}\) where each \(A_i\) is normed, the \emph{final structure} on \(B\) is given by
\begin{displaymath}
  |b|=\bigvee\{|a|\mid a\in A_i,\, g_i(a)=b, i\in I\}.
\end{displaymath}

The strict monoidal functor \(\mO \colon\SetV\to\Set\) induces the 2-functors
\begin{displaymath}
  \cb{\mO} \colon\CatV \longrightarrow\Cat
  \quad\text{and}\quad
    \cb{\mO} \colon\CATV \longrightarrow\CAT
\end{displaymath}
which forget the norm on each hom-set. Moreover, both functors are strict monoidal: for normed categories \(\ncatA\) and \(\ncatB\),
\begin{displaymath}
  \cb{\mO}(\ncatE)=\catE
  \quad\text{and}\quad
  \cb{\mO}(\ncatA\otimes\ncatB)
  =\cb{\mO}(\ncatA)\times\cb{\mO}(\ncatB).
\end{displaymath}

Following \citet{Kel82}, we also consider the representable 2-functors
\begin{displaymath}
  (-)_{\circ}:=(\CatV)(\ncatE,-)\colon\CatV \longrightarrow \Cat
  \quad\text{and}\quad
    (-)_{\circ}:=(\CATV)(\ncatE,-)\colon\CATV \longrightarrow \CAT
\end{displaymath}
which send a $\mcatV$-normed category $\ncatA$ to the category $\ncatA_{\circ}$ with the same objects as $\ncatA$, but with only those morphisms $f \colon a \to b$ in $\ncatA$ with $\kk\leq |f|$. We note that \((-)_{\circ}\) is induced by the symmetric lax monoidal functor
\begin{displaymath}
  \SetV(E,-) \colon\SetV \longrightarrow\Set
\end{displaymath}
which sends the \(\mcatV\)-normed set \(A\) to the set
\begin{math}
  \{a\in A\mid \kk\leq|a|\}
\end{math}
and \(\mcatV\)-normed maps to the corresponding restrictions.

\begin{example}\label{w:ex:1}
  The symmetric strict monoidal functor
  \begin{align*}
    \ms \colon \SetV
    & \longrightarrow\mcatV,\\
    f \colon A\to B
    & \longmapsto
      \bigvee_{a\in A}|a|\leq \bigvee_{b\in B}|b|
  \end{align*}
  induces the symmetric strict monoidal 2-functors
  \begin{displaymath}
    \cb{\ms} \colon\CatV\longrightarrow\Cats{\mcatV}
    \quad\text{and}\quad
    \cb{\ms} \colon\CATV\longrightarrow\CATs{\mcatV},
  \end{displaymath}
  sending a \(\mcatV\)-normed category \(\ncatA\) to the \(\mcatV\)-category \(\cb{\ms}\ncatA\) with the same set of objects and with
  \begin{displaymath}
    (\cb{\ms}\ncatA)(a,b)=\bigvee_{f \colon a\to b}|f|.
  \end{displaymath}
  The functor \(\ms \colon \SetV\to\mcatV\) has a right adjoint
  \begin{displaymath}
    \mi \colon\mcatV \longrightarrow\SetV,\quad
    v \longmapsto (\{\star\},|{\star}|=v)
  \end{displaymath}
  which is symmetric strong monoidal and induces the corresponding right adjoint functors
  \begin{displaymath}
    \cb{\mi} \colon\Cats{\mcatV}\longrightarrow\CatV
    \quad\text{and}\quad
    \cb{\mi} \colon\CATs{\mcatV}\longrightarrow\CATV
  \end{displaymath}
  which interpret the \(\mcatV\)-category \(X\) as the \(\mcatV\)-normed category \(\cb{\mi}X\) with exactly one arrow \(x\to y\) for each pair \((x,y)\) of objects, normed by
  \begin{displaymath}
    |(x,y)|=X(x,y).
  \end{displaymath}
\end{example}

\subsection{$\mcatW$ as a $\mcatW$-category}
\label{sec:mcatw-as-mcatw}

\emph{From now on we assume that our symmetric monoidal category \(\mcatW\) is also closed}, and we denote the right adjoint to \(-\otimes A\) by \([A,-]\). Then the monoidal closed category \(\mcatW\) gives rise to a \(\mcatW\)-category, also denoted by \(\mcatW\), with the same set of objects as \(\mcatW\) and, for all objects \(A\) and \(B\) of \(\mcatW\),
\begin{displaymath}
  \mcatW(A,B)=[A,B].
\end{displaymath}
Here the unit arrow \(E\to\mcatW(A,A)\) corresponds to the left unit arrow \(E\otimes A\to A\), and the composition arrow \(\mcatW(B,C)\times\mcatW(A,B)\to\mcatW(A,C)\) to the composite
\begin{displaymath}
  \mcatW(B,C)\otimes\mcatW(A,B)\otimes A
  \longrightarrow
  \mcatW(B,C)\otimes B
  \longrightarrow A.
\end{displaymath}

\begin{example}\label{d:ex:3}
  Recall that all categories of Example~\ref{d:ex:1} are closed. It is important to note that, for \(\mcatW=\SetV\), the internal hom $[A,B]$ has carrier set $\Set(A,B)$ (\emph{all mappings $\varphi \colon A \to B$}), with their norm defined by
  \begin{displaymath}
    |\varphi|=\bigwedge_{a\in A}[|a|,|\varphi a|].
  \end{displaymath}
\end{example}

\begin{example}
  The monoidal category \(\SetV\) becomes a \(\mcatV\)-normed category whose objects are $\mcatV$-normed sets, but whose normed hom-sets of morphisms $A\to B$ are given by the internal hom $[A,B]$ of $\SetV$, {\em i.e.}\ by \(\Set(A,B)\). We write \(\nSetss{\mcatV}\) to denote this normed category and note
  \begin{displaymath}
    (\nSetss{\mcatV})_\circ = \Setss{\mcatV}
    \quad\text{and}\quad
    \cb{\mO}(\nSetss{\mcatV})\simeq\Set.
  \end{displaymath}
\end{example}

\begin{theorem}
  \label{d:thm:1}
  Let \(\Gamma \colon\mcatW\to\mcatU\) be a symmetric lax monoidal functor between closed symmetric monoidal categories. Then \(\Gamma\) induces an \(\mcatU\)-functor \(\gamma \colon\cb{\Gamma}(\mcatW)\to\mcatU\) where
  \begin{enumerate}
  \item \(\gamma(A)=\Gamma(A)\) for every object \(A\) of \(\cb{\Gamma}\mathcal{W}\),
  \item for all objects \(A\) and \(B\) of \(\cb{\Gamma}\mathcal{W}\), the arrow
    \begin{displaymath}
      \gamma_{A,B}\colon\cb{\Gamma}(A,B)
      =\Gamma(\mcatW(A,B))\longrightarrow\mcatU(\Gamma A,\Gamma B)
    \end{displaymath}
    in \(\mcatU\) corresponds to the composite of the canonical maps
    \begin{displaymath}
      \Gamma(\mcatW(A,B))\otimes \Gamma A
      \xrightarrow{\quad\tau\quad}
      \Gamma(\mcatW(A,B)\otimes A)
      \xrightarrow{\quad\Gamma\ev\quad}
      \Gamma B.
    \end{displaymath}
  \end{enumerate}
\end{theorem}
\begin{proof}
  By \citep[Proposition~6.4.5]{Bor94a}, the symmetric monoidal functor  \(\Gamma \colon\mcatW\to\mcatU\) is also \emph{closed}, then apply \citep[Theorem~6.6 of Chapter~1]{EK66}. However, for the convenience of the reader, we give here a direct computation. Firstly, the commutativity of the diagram
  \begin{displaymath}
    \begin{tikzcd}[column sep=large]
      \Gamma(\mathcal{W}(A,A)) %
      \ar{r}{\gamma_{A,A}} %
      & \mathcal{U}(\Gamma A,\Gamma A)\\
      \Gamma(E_{\mathcal{W}}) %
      \ar{u}{\Gamma(\unit)} %
      & E_{\mathcal{U}} %
      \ar{u}[swap]{\unit} %
      \ar{l}{\unit}
    \end{tikzcd}
  \end{displaymath}
  follows, by adjunction, from the commutativity of the outer diagram below.
  \begin{displaymath}
    \begin{tikzcd}[column sep=large,row sep=large]
      \Gamma(\mathcal{W}(A,A))\otimes\Gamma A %
      \ar{r} %
      & \Gamma(\mathcal{W}(A,A)\otimes A) %
      \ar{r}{\Gamma\ev} %
      & \Gamma A\\
      & \Gamma(E_{\mathcal{W}}\otimes A) %
      \ar[dashed]{u}{\Gamma(\unit\otimes 1)} %
      \ar[dashed]{ur}[swap]{\Gamma(\text{left unit})}\\
      \Gamma E_{\mathcal{W}}\otimes \Gamma A %
      \ar{uu}{\Gamma(\unit)\otimes 1} %
      \ar[dashed]{ur}
      && E_{\mathcal{U}}\otimes \Gamma A %
      \ar{uu}[swap]{\text{left unit}} %
      \ar{ll}{\unit\otimes 1}
    \end{tikzcd}
  \end{displaymath}
  Here the left hand diagram commutes because the family of arrow \(\Gamma X\otimes \Gamma Y\to \Gamma(X\otimes Y)\) is a natural transformation, the upper right diagram commutes by definition of the unit \(E_{\mathcal{W}}\to \mcatW(A,A)\), and the lower right diagram commutes because \(\Gamma\) is lax monoidal.

  Secondly, we have to show that the diagram
  \begin{displaymath}
    \begin{tikzcd}[column sep=large]
      \Gamma(\mcatW(B,C))\otimes\Gamma(\mcatW(A,B)) %
      \ar{r}{\tau}\ar{d}{\gamma\otimes\gamma} %
      & \Gamma(\mcatW(B,C)\otimes\mcatW(A,B)) %
      \ar{r}{\Gamma(\comp)} %
      & \Gamma(\mcatW(A,C)) %
      \ar{d}{\gamma}\\
      \mcatU(\Gamma B,\Gamma C)\otimes \mcatU(\Gamma A,\Gamma B) %
      \ar{rr}[swap]{\comp} && \mcatU(\Gamma A,\Gamma C)
    \end{tikzcd}
  \end{displaymath}
  is commutative which, by adjunction, follows from the commutativity of the ``solid'' diagram below where we assume that the associativity arrows are identities (see also \citep[Theorem~XI.3.1]{Mac98}).
  \begin{displaymath}\tiny
    \begin{tikzcd}[row sep=large, column sep=tiny]
      & \Gamma(\mcatW(B,C)\otimes\mcatW(A,B))\otimes \Gamma A %
      \ar{rr}{\Gamma(\comp)\otimes 1} %
      \ar[dashed]{d}{\tau} %
      && \Gamma(\mcatW(A,C))\otimes \Gamma A %
      \ar[bend left=15]{rdd}{\gamma\otimes 1}
      \ar[dashed]{d}{\tau} \\
      & \Gamma(\mcatW(B,C)\otimes\mcatW(A,B)\otimes A) %
      \ar[dashed]{rdd}{\Gamma(1\otimes\ev)} \ar[dashed]{rr}{\Gamma(\comp\otimes 1)} %
      && \Gamma(\mcatW(A,C)\otimes A) %
      \ar[dashed, bend right=15]{rdd}[swap]{\Gamma\ev}\\
      \Gamma(\mcatW(B,C))\otimes \Gamma(\mcatW(A,C))\otimes \Gamma A %
      \ar[bend left=15]{ddr}[near end]{1\otimes\gamma\otimes 1} %
      \ar[dashed]{d}[swap]{1\otimes\tau} %
      \ar[bend left=15]{uur}{\tau\otimes 1} %
      &&&& \mcatU(\Gamma A,\Gamma C)\otimes \Gamma A %
      \ar{d}{\ev}\\
      \Gamma(\mcatW(B,C))\otimes\Gamma(\mcatW(A,B)\otimes A) %
      \ar[dashed, bend right=15]{ddr}[swap]{1\otimes \Gamma\ev} %
      \ar[dashed, bend right=15]{uur}{\tau} && \Gamma(\mcatW(B,C)\otimes B) %
      \ar[dashed]{rr}{\Gamma\ev} %
      && \Gamma C\\
      & \Gamma(\mcatW(B,C))\otimes\mcatU(\Gamma A,\Gamma B)\otimes\Gamma A %
      \ar[dashed]{d}[swap]{1\otimes\ev} \ar{rr}{\gamma\otimes 1\otimes 1} &&\mcatU(\Gamma B,\Gamma C)\otimes\mcatU(\Gamma A,\Gamma B)\otimes \Gamma A %
      \ar[dashed]{d}{1\otimes\ev} %
      \ar[bend right=15]{uur}[near start]{\comp\otimes 1}\\
      & \Gamma(\mcatW(B,C))\otimes \Gamma B %
      \ar[dashed,bend right=15]{uur}[swap]{\tau} %
      \ar[dashed]{rr}[swap]{\gamma\otimes 1} %
      && \mcatU(\Gamma B,\Gamma C)\otimes \Gamma B %
      \ar[dashed, bend right]{uur}[swap]{\ev}
    \end{tikzcd}
  \end{displaymath}
  In fact, as indicated in the diagram above, we obtain
  \begin{multline*}
    \ev\cdot(\gamma\otimes 1)\cdot(\Gamma(\comp)\otimes 1)\cdot(\tau\otimes 1)\\
    \begin{aligned}
      &=\Gamma(\ev)\cdot\tau\cdot(\Gamma(\comp)\otimes 1)\cdot(\tau\otimes 1)
      && \text{[definition of \(\gamma\)]}\\
      &=\Gamma(\ev)\cdot\Gamma(\comp\otimes 1)\cdot\tau\cdot(\tau\otimes 1)
      &&\text{[naturality of \(\tau\)]}\\
      &=\Gamma(\ev)\cdot\Gamma(\comp\otimes 1)\cdot\tau\cdot(1\otimes\tau)
      &&\text{[\(\Gamma\) is lax monoidal]}\\
      &=\Gamma(\ev)\cdot\Gamma(1\otimes\ev)\cdot\tau\cdot(1\otimes\tau) && \text{[definition of \(\comp\)]}\\
      &=\Gamma(\ev)\cdot\tau\cdot(1\otimes\Gamma(\ev))\cdot(1\otimes\tau)
      && \text{[naturality of \(\tau\)]}\\
      &=\ev\cdot(\gamma\otimes 1)\cdot(1\otimes\Gamma(\ev))\cdot(1\otimes\tau)
      && \text{[definition of \(\gamma\)]}\\
      &=\ev\cdot(\gamma\otimes 1)\cdot(1\otimes\ev)\cdot(1\otimes\gamma \otimes 1)
      && \text{[definition of \(\gamma\)]]}\\
      & =\ev\cdot(\gamma\otimes\ev)\cdot (1\otimes\gamma \otimes 1)\\
      & =\ev\cdot(1\otimes\ev)\cdot(\gamma\otimes 1\otimes 1)\cdot (1\otimes\gamma \otimes 1)\\
      & =\ev\cdot(\comp\otimes 1)\cdot(\gamma\otimes 1\otimes 1)\cdot (1\otimes\gamma \otimes 1)
      && \text{[definition of \(\comp\)]}.
    \end{aligned}
  \end{multline*}
\end{proof}

\begin{example}
  Let \(\mcatV\) be a quantale and consider the symmetric strong monoidal functor \(\mi \colon\mcatV\to\SetV\) and the symmetric strict monoidal functor \(\ms \colon\SetV\to\mcatV\) of Example \ref{w:ex:1}. Then \(\cb{\mi}(\mcatV)\) is the \(\mcatV\)-normed category with objects the elements of \(\mcatV\), and for all \(u,v\in\mcatV\),
  \begin{displaymath}
    \cb{\mi}(\mcatV)(u,v)=(\{\star\},[u,v]).
  \end{displaymath}
  Furthermore, the \(\mcatV\)-normed functor
  \begin{displaymath}
    \cb{\mi}(\mcatV) \longrightarrow\nSetV
  \end{displaymath}
  sends \(u\in\mcatV\) to \((\{\star\},u)\) and, for \(u,v\in\mcatV\), the map
  \begin{displaymath}
    \cb{\mi}(\mcatV)(u,v) \longrightarrow\nSetV((\{\star\},u),(\{\star,v\}))
  \end{displaymath}
  sends \(\star\) to the unique map \(\{\star\}\to\{\star\}\). This map is indeed \(\mcatV\)-normed since, for all \(u,v\in\mcatV\),
  \begin{displaymath}
    \nSetV((\{\star\},u),(\{\star\},v))=[u,v].
  \end{displaymath}
  The elements of the (large) \(\mcatV\)-category \(\cb{\ms}(\nSetV)\) are the normed sets, and
  \begin{displaymath}
    \cb{\ms}(\nSetV)(A,B)=\bigvee_{f \colon A\to B}|f|
  \end{displaymath}
  for all normed sets \(A\) and \(B\). The \(\mcatV\)-functor \(\cb{\ms}(\nSetV)\to\mcatV\) sends the normed set \(A\) to \(\bigvee_{a\in A}|a|\).
\end{example}

\subsection{Extraordinary naturality}
\label{sec:extr-natur}

Next we recall the notion of \emph{(extraordinary) \(\mcatW\)-naturality} (see also \citep{Dub70}), for a symmetric monoidal category \(\mcatW\) and for certain families of \(\mcatW\)-arrows with respect to a \(\mcatW\)-functor \(T \colon \mathcal{A}^{\op}\otimes \mathcal{A}\to\mcatW\). First note that \(\mcatW\)-functoriality of \(T\) gives us arrows
\begin{displaymath}
  T(-,-)\colon(\mathcal{A}^{\op}\otimes \mathcal{A})((A,B),(A',B'))
  \longrightarrow [T(A,B),T(A',B')]
\end{displaymath}
in \(\mcatW\) which can equivalently be written as
\begin{displaymath}
  T(-,-)\colon \mathcal{A}(A',A)\otimes T(A,B)\otimes \mathcal{A}(B,B')
  \longrightarrow T(A',B').
\end{displaymath}
Considering \(A=A'\) and using the identity law, we obtain the arrows
\begin{displaymath}
  T(A,-)\colon T(A,B)\otimes \mathcal{A}(B,B')
  \longrightarrow T(A,B')
\end{displaymath}
and, equivalently,
\begin{displaymath}
  T(A,-)\colon \mathcal{A}(B,B')\longrightarrow
  [T(A,B),T(A,B')]
\end{displaymath}
in \(\mcatW\) which identify \(T(A,-)\) as a \(\mcatW\)-functor \(T(A,-)\colon \mathcal{A}\longrightarrow\mcatW\). Similarly, for \(B=B'\) we obtain
\begin{displaymath}
   T(-,B)\colon \mathcal{A}(A',A)\otimes T(A,B)
  \longrightarrow T(A',B)
\end{displaymath}
which makes \(T(-,B)\) a \(\mcatW\)-functor \(T(-,B)\colon \mathcal{A}^{\op}\longrightarrow\mcatW\). A family \((\gamma_A \colon T(A,A)\to K)_A\), with \(K\) fixed, is \emph{\(\mcatW\)-natural} whenever, for all objects \(A,B\) in \(\mathcal{A}\), the diagram
\begin{equation}
  \label{d:eq:1}
  \begin{tikzcd}[row sep=large, column sep=large]
    \mathcal{A}(A,B) %
    \ar{r}{T(B,-)} %
    \ar{d}[swap]{T(-,A)} %
    & {[T(B,A),T(B,B)]} %
    \ar{d}{[T(B,A),\gamma_{B}]} \\
    {[T(B,A),T(A,A)]} %
    \ar{r}[swap]{[T(B,A),\gamma_{A}]} %
    & {[T(B,A),K]} %
  \end{tikzcd}
\end{equation}
in \(\mcatW\) commutes. Using the adjunction \(-\otimes T(B,A)\dashv [T(B,A),-]\) and symmetry of the tensor in \(\mcatW\), the square \eqref{d:eq:1} commutes if and only if
\begin{equation}
  \label{d:eq:2}
  \begin{tikzcd}[row sep=large, column sep=large]
    \mathcal{A}(A,B)\otimes T(B,A) %
    \ar{r}{T(B,-)} %
    \ar{d}[swap]{T(-,A)} %
    & T(B,B) %
    \ar{d}{\gamma_{B}} \\
    T(A,A) %
    \ar{r}[swap]{\gamma_{A}} %
    & K %
  \end{tikzcd}
\end{equation}
commutes in \(\mcatW\). A universal (with respect to the commutativity of \eqref{d:eq:2}) \(\mcatW\)-natural family \((\gamma_A \colon T(A,A)\to C)\) is called a \emph{coend} of \(T\). If it exists, the coend of \(T\) is unique and one writes
\begin{displaymath}
  \int^{A\in\ob\mathcal{A}}T(A,A)
  \quad\text{or simply}\quad
  \int^{\mathcal{A}}T
\end{displaymath}
for the object \(C\). If the category \(\mcatW\) is cocomplete, then every functor \(T \colon \mathcal{A}^{\op}\otimes \mathcal{A}\to\mcatW\) with \(\mathcal{A}\) small admits a coend: by \eqref{d:eq:2}, it can be constructed as the coequaliser
\begin{displaymath}
  \begin{tikzcd}
    \sum_{A,B\in\ob \mathcal{A}}\mathcal{A}(A,B)\otimes T(B,A) %
    \ar[shift left=3pt]{r} %
    \ar[shift right=3pt]{r} %
    & \sum_{A\in\ob \mathcal{A}}T(A,A) %
    \ar{r}{\gamma} & \int^{\mathcal{A}}T
  \end{tikzcd}
\end{displaymath}
in \(\mcatW\).

Dually, a family \((\beta_A \colon K\to T(A,A))_{A\in\ob \mathcal{A}}\) is \(\mcatW\)-natural whenever, for all objects \(A,B\) in \(\mcatW\), the diagram
\begin{equation}
  \label{d:eq:3}
  \begin{tikzcd}[row sep=large, column sep=large]
    \mathcal{A}(A,B) %
    \ar{r}{T(A,-)} %
    \ar{d}[swap]{T(-,B)} %
    & {[T(A,A),T(A,B)]} %
    \ar{d}{[\beta_{A},T(A,B)]} \\
    {[T(B,B),T(A,B)]} %
    \ar{r}[swap]{[\beta_{B},T(A,B)]} %
    & {[K,T(A,B)]} %
  \end{tikzcd}
\end{equation}
commutes. Using the adjunction, \(([-,T(A,B)]\colon\mcatW\to\mcatW^{\op}) \dashv ([-,T(A,B)]\colon\mcatW^{\op}\to\mcatW)\), the square~\eqref{d:eq:3} commutes if and only if
\begin{displaymath}
  \begin{tikzcd}[row sep=large, column sep=large]
    K %
    \ar{r}{\beta_{A}} %
    \ar{d}[swap]{\beta_{B}} %
    & T(A,A) %
    \ar{d}{T(A,-)} \\
    T(B,B) %
    \ar{r}[swap]{T(-,B)} %
    & {[\mathcal{A}(A,B),T(A,B)]} %
  \end{tikzcd}
\end{displaymath}
commutes. A universal \(\mcatW\)-natural family \((\lambda_A \colon L\to T(A,A))_{A\in\ob \mathcal{A}}\) is called an \emph{end} for \(T\), and one writes
\begin{displaymath}
  \int_{A\in\ob\mathcal{A}}T(A,A)
  \quad\text{or simply}\quad
  \int_{\mathcal{A}}T
\end{displaymath}
for the object \(L\). If the category \(\mcatW\) is complete, then every functor \(T \colon \mathcal{A}^{\op}\otimes \mathcal{A}\to\mcatW\) with \(\mathcal{A}\) small admits an end: it can be constructed as the equaliser
\begin{displaymath}
  \begin{tikzcd}
    \int_{\mathcal{A}}T\ar{r}{\lambda} %
    & \prod_{A\in\ob\mathcal{A}}T(A,A) %
    \ar[shift left=3pt]{r} %
    \ar[shift right=3pt]{r} %
    & \prod_{A,B\in\ob\mathcal{A}}[\mathcal{A}(A,B),T(A,B)]
  \end{tikzcd}
\end{displaymath}
in \(\mcatW\).

\begin{remark}
  \label{d:rem:3}%
  The notion of \(\mcatW\)-natural family is self-dual in the following sense. For a \(\mcatW\)-functor \(T \colon\mathcal{A}^{\op}\otimes \mathcal{A}\to\mcatW\), we may consider the functor \(T^s \colon{\mathcal{A}^{\op}}^{\op}\otimes \mathcal{A}^{\op}\to\mcatW\) defined by ``swapping the arguments''.
  \begin{displaymath}
    \begin{tikzcd}
      {\mathcal{A}^{\op}}^{\op}\otimes \mathcal{A}^{\op}
      =\mathcal{A}\otimes \mathcal{A}^{\op}
      \ar{r}
      \ar[bend left]{rr}{T^s}
      & \mathcal{A}^{\op}\otimes \mathcal{A}
      \ar{r}[swap]{T}
      & \mcatW
    \end{tikzcd}
  \end{displaymath}
  A family \((\gamma_A \colon T^s(A,A)=T(A,A)\to K)_A\) is \(\mcatW\)-natural for \(T\) is and only if it is so for \(T^s\), in particular, there exists a coend of \(T\) if and only if there exists an coend of \(T^s\) and, moreover,
  \begin{displaymath}
    \int^{\mathcal{A}}T\cong\int^{\mathcal{A}^{\op}}T^s.
  \end{displaymath}
  Similarly, an end of \(T\) is also an end of \(T^s\) and vice versa.
  \end{remark}

\begin{example}
  \begin{enumerate}
  \item For \(\mcatW=\mcatV\) a quantale and a \(\mcatV\)-functor \(t \colon X^{\op}\otimes X\to\mcatV\), a \(\mcatV\)-natural family \((u\to t(x,x))_{x\in X}\) is determined by a lower bound \(u\) of \(\{t(x,x)\mid x\in X\}\), hence an end for \(t\) is given by an infimum of \(\{t(x,x)\mid x\in X\}\). Likewise, a coend for \(t\) is given by a supremum of \(\{t(x,x)\mid x\in X\}\).
  \item Consider now \(\mcatW=\Set\) and a functor \(T \colon\catX^{\op}\times\catX\to\Set\). Then a family \((\beta_{x}\colon K\to T(x,x))_{x\in\ob\catX}\) is \(\Set\)-natural if and only if, for all \(f \colon x\to y\) in \(\catX\), the diagram
    \begin{equation}
      \label{d:eq:4}
      \begin{tikzcd}[row sep=large, column sep=large]
        K %
        \ar{r}{\beta_{x}} %
        \ar{d}[swap]{\beta_{y}} %
        & T(x,x) %
        \ar{d}{T(x,f)} \\
        T(y,y) %
        \ar{r}[swap]{T(f,y)} %
        & T(x,y) %
      \end{tikzcd}
    \end{equation}
    commutes, that is, \((\beta_{x}\colon K\to T(x,x))_{x\in\ob\catX}\) is a \emph{wedge} (see \citep{Mac98}). Similarly, a family \((\gamma_x \colon T(x,x)\to K)_{x\in\ob\catX}\) is \(\Set\)-natural if and only if it is a wedge. Therefore the notion of \(\Set\)-enriched (co)end coincides with the usual notion of (co)end for ordinary categories.
  \item Let us now consider the case \(\mcatW=\SetV\) and let \(T \colon\ncatX^{\op}\otimes\ncatX\to\nSetV\) be a normed functor. Similarly to the \(\Set\)-case, a family \((\beta_{x}\colon K\to T(x,x))_{x\in\ob\ncatX}\) of arrows in \(\SetV\) is \(\SetV\)-natural if and only if the square~\eqref{d:eq:4} commutes in \(\SetV\), for all \(f \colon x\to y\) in \(\ncatX\). Dually, a family \((\gamma_x\colon T(x,x)\to K)_{x\in\ob\ncatX}\) is \(\SetV\)-natural if and only if, for all \(f \colon x\to y\),
    \begin{displaymath}
      \begin{tikzcd}[row sep=large, column sep=large]
        T(y,x) %
        \ar{r}{T(y,f)} %
        \ar{d}[swap]{T(f,x)} %
        & T(y,y) %
        \ar{d}{\gamma_{y}} \\
        T(x,x) %
        \ar{r}[swap]{\gamma_x} %
        & K %
      \end{tikzcd}
    \end{displaymath}
    commutes.
  \end{enumerate}
\end{example}

Recall that the functor \(\mO \colon\SetV\to\Set\) is topological (see Proposition~\ref{d:prop:2}), which facilitates the understanding of (co)ends for \(\mcatV\)-normed categories. We note first the following fact.

\begin{lemma}
  Let \((c_i \colon X_i\to X)_{i\in I}\) be an \(\mO\)-final cocone and \((p_j \colon Y\to Y_j)_{j\in J}\) be an \(\mO\)-initial cone in \(\SetV\), \(Y\) be a normed set, \(f \colon X\to Y\) be a map and \((f_i \colon X_i\to Y)_{i\in I}\) and \((f_j \colon X\to Y_j)_{j\in J}\) families of maps so that the diagrams
  \begin{displaymath}
    \begin{tikzcd}
      X\ar{r}{f} & Y\\
      X_i\ar{u}{c_i}\ar{ur}[swap]{f_i}
    \end{tikzcd}
    \qquad\qquad
    \begin{tikzcd}
      X\ar{r}{f}\ar{rd}[swap]{g_j} %
      & Y\ar{d}{p_j}\\ & Y_j
    \end{tikzcd}
  \end{displaymath}
  commute for every \(i\in I\) and \(j\in J\). Then, in the \(\mcatV\)-normed category \(\nSetV\), one has
  \begin{displaymath}
    |f|=\bigwedge_{i\in I}|f_i|
    \quad\text{and}\quad
    |f|=\bigwedge_{j\in J}|g_j|.
  \end{displaymath}
\end{lemma}
\begin{proof}
  To see the first formula, we calculate
  \begin{align*}
    |f|
    &=\bigwedge_{x\in X}[|x|,|f(x)|]\\
    &=\bigwedge_{x\in X}
      [\bigvee_{i\in I}\bigvee_{x'\in c_i^{-1}(x)}|x'|,|f(x)|]\\
    &=\bigwedge_{i\in I}\bigwedge_{x\in X}
      \bigwedge_{x'\in c_i^{-1}(x)}[|x'|,|f(x)|]\\
    &=\bigwedge_{i\in I}\bigwedge_{x'\in X_i}[|x'|,|f_i(x')|]
    =\bigwedge_{i\in I}|f_i|.
  \end{align*}
  The proof of the second one is similar, in fact, easier.
\end{proof}

\begin{proposition}
  \label{d:prop:1}
  Let \(T \colon\ncatX^{\op}\otimes\ncatX\to\nSetV\) be a normed functor. Then there exists an end (respectively a coend) for \(T\) if and only if there exists an end (respectively a coend) for the functor
  \begin{displaymath}
    T^{\mO}\colon\cb{\mO}(\ncatX)^{\op}\times\cb{\mO}(\ncatX)
    =\cb{\mO}(\ncatX^{\op}\otimes\ncatX)
    \xrightarrow{\quad\cb{\mO}(T)\quad}\cb{\mO}(\nSetV)
    \xrightarrow[\text{Theorem~\ref{d:thm:1}}]{\gamma}\Set.
  \end{displaymath}
  Moreover, an end \((\int_{\ncatX}T\to T(x,x))_{x\in\ob\ncatX}\) for \(T\) may be constructed by first taking an end \((\lambda_x \colon\int_{\cb{\mO}(\ncatX)}T^{\mO}\to \mO(T(x,x)))_{x\in\ob\ncatX}\) for \(T^{\mO}\) and then equipping \(\int_{\cb{\mO}(\ncatX)}T^{\mO}\) with the initial structure with respect to the forgetful functor \(\mO \colon\SetV\to\Set\). Dually, a coend \((T(x,x)\to\int^{\ncatX}T)_{x\in\ob\ncatX}\) for \(T\) may be constructed by first taking a coend \((\mO(T(x,x))\to \int^{\cb{\mO}(\ncatX)}T^{\mO})_{x\in\ob\ncatX}\) for \(T^{\mO}\) and then equipping \(\int^{\cb{\mO}(\ncatX)}T^{\mO}\) with the final structure with respect to the forgetful functor \(\mO \colon\SetV\to\Set\). In particular,
  we have
  \begin{displaymath}
    \cb{\mO}\left(\int^{\ncatX}T\right)=\int^{\cb{\mO}(\ncatX)}T^{\mO}
    \quad\text{and}\quad
    \cb{\mO}\left(\int_{\ncatX}T\right)=\int_{\cb{\mO}(\ncatX)}T^{\mO}.
  \end{displaymath}
\end{proposition}

Finally, for \(\mcatW\) complete and \(\mcatW\)-functors \(F,G \colon \mathcal{A}\to \mathcal{B}\), one puts
\begin{equation}
  \label{d:eq:5}
  [\mathcal{A},\mathcal{B}](F,G)=\int_{A\in\ob \mathcal{A}}\mathcal{B}(FA,GA)
\end{equation}
whenever the end exists. The elements
\begin{displaymath}
  E \longrightarrow [\mathcal{A},\mathcal{B}](F,G)
\end{displaymath}
correspond precisely to \(\mcatW\)-natural families
\begin{displaymath}
  (E \longrightarrow \mathcal{B}(FA,GA))_{A\in\ob \mathcal{A}}
\end{displaymath}
of \(\mathcal{B}(FA,GA))_{A\in\ob \mathcal{A}}\), that is, to \(\mcatW\)-natural transformations \(F\to G\). The end in \eqref{d:eq:5} certainly exists if \(\mathcal{A}\) is small, and in this case \([\mathcal{A},\mathcal{B}](F,G)\) is the hom-object of the \(\mcatW\)-category of all \(\mcatW\)-functors \(\mathcal{A}\to \mathcal{B}\), moreover, this construction defines the right adjoint to
\begin{displaymath}
  -\otimes \mathcal{A}\colon\Cats{\mcatW}\to\Cats{\mcatW}
  \quad\text{and}\quad
  -\otimes \mathcal{A}\colon\CATs{\mcatW}\to\CATs{\mcatW}.
\end{displaymath}

\begin{example}
  For \(\mcatW=\mcatV\) a quantale and for \(\mcatV\)-functors \(f,g \colon X\to Y\),
  \begin{displaymath}
    [X,Y](f,g)=\bigwedge_{x\in X}Y(fx,gx).
  \end{displaymath}
\end{example}

\begin{example}
  For \(\mcatW=\Set\) and for functors \(F,G \colon\catA\to\catX\) (with \(\catA\) small), \([\catA,\catX](F,G)\) is the set of all natural transformations \(F\to G\).
\end{example}

\begin{example}
  For \(\mcatW=\SetV\) and for \(\mcatV\)-normed functors \(F,G \colon\ncatA\to\ncatX\) (with \(\ncatA\) small), from Proposition~\ref{d:prop:1} we obtain that \([\ncatA,\ncatX](F,G)\) is the set of \emph{all} natural transformations \(\alpha \colon F\to G\), normed by
  \begin{displaymath}
    |\alpha|=\bigwedge_{A\in\ob\ncatA}|\alpha_A|.
  \end{displaymath}
\end{example}

\subsection{Enriched distributors}
\label{sec:enrich-distr}

We assume here that \(\mcatW\) is a closed symmetric monoidal complete and cocomplete category. For \(\mcatW\)-categories \(\mathcal{A}\) and \(\mathcal{B}\), a \emph{\(\mcatW\)-distributor} (also \(\mcatW\)-profunctor or \(\mcatW\)-bimodule) \(\Phi \colon \mathcal{A}\modto \mathcal{B}\) is a \(\mcatW\)-functor
\begin{displaymath}
  \Phi \colon \mathcal{A}^{\op}\otimes \mathcal{B}
  \longrightarrow\mcatW.
\end{displaymath}
In particular, for a \(\mcatW\)-functor \(F \colon \mathcal{A}\to \mathcal{B}\), we have the \(\mcatW\)-distributors \(F_{*}\colon \mathcal{A}\modto \mathcal{B}\) and \(F^{*}\colon \mathcal{B}\modto \mathcal{A}\) defined by
\begin{displaymath}
  \mathcal{A}^{\op}\otimes \mathcal{B}
  \xrightarrow{\quad F^{\op}\otimes 1\quad}
  \mathcal{B}^{\op}\otimes \mathcal{B}
  \xrightarrow{\quad\mathcal{B}(-,-)\quad}
  \mcatW
\end{displaymath}
and
\begin{displaymath}
  \mathcal{B}^{\op}\otimes \mathcal{A}
  \xrightarrow{\quad 1\otimes F\quad}
  \mathcal{B}^{\op}\otimes \mathcal{B}
  \xrightarrow{\quad\mathcal{B}(-,-)\quad}
  \mcatW,
\end{displaymath}
respectively. We write \(\Dists{\mcatW}(\mathcal{A},\mathcal{B})\) for the category of \(\mcatW\)-distributors \(\mathcal{A}\modto \mathcal{B}\) and \(\mcatW\)-natural transformations between them.

\begin{remark}
  \label{d:rem:1}
  Recall from Subsection~\ref{sec:change-base} that every symmetric lax monoidal functor \(\Gamma \colon\mcatW\to\mcatU\) induces the change-of-base functor \(\cb{\Gamma}\colon\Cats{\mcatW}\to\Cats{\mcatU}\). Now we show that this correspondence can to some extent be lifted to distributors. For \(\mcatW\)-categories \(\mathcal{A}\) and \(\mathcal{B}\), we define the functor
  \begin{displaymath}
    \db{\Gamma}\colon\Dists{\mcatW}(\mathcal{A},\mathcal{B})
    \longrightarrow
    \Dists{\mcatU}
    (\cb{\Gamma}(\mathcal{A}),\cb{\Gamma}(\mathcal{B}))
  \end{displaymath}
  which sends \(\Phi \colon \mathcal{A}^{\op}\otimes \mathcal{B}\to\mcatW\) to the \(\mcatU\)-distributor \(\db{\Gamma}(\Phi)=\gamma\,\cb{\Gamma}(\Phi)\,\tau\) (noting that \(\cb{\Gamma}(\mathcal{A}^{\op})=\cb{\Gamma}(\mathcal{A})^{\op}\)),
  \begin{displaymath}
    \begin{tikzcd}
      \cb{\Gamma}(\mathcal{A})^{\op}\otimes\cb{\Gamma}(\mathcal{B})
      \ar{r}{\tau}\ar[bend left=20]{rrrr}{\db{\Gamma}(\Phi)}
      & \cb{\Gamma}(\mathcal{A}^{\op}\otimes \mathcal{B})
      \ar{r}{\cb{\Gamma}(\Phi)}
      & \cb{\Gamma}(\mcatW)
      \ar{rr}{\gamma}[swap]{\text{Theorem~\ref{d:thm:1}}}
      && \mcatU
    \end{tikzcd}
  \end{displaymath}
  that is, \(\db{\Gamma}(\Phi)(A,B)=\Gamma(\Phi(A,B))\), and \(\db{\Gamma}\) sends \(\alpha \colon\Phi\to\Psi\) in \(\Dists{\mcatW}(\mathcal{A},\mathcal{B})\) to
  \begin{displaymath}
    \db{\Gamma}(\alpha)=\gamma\,\cb{\Gamma}(\alpha)\,\tau
    \colon\db{\Gamma}(\Phi) \longrightarrow\db{\Gamma}(\Psi),
  \end{displaymath}
  here functoriality of \(\db{\Gamma}\) follows from the interchange laws of Subsection~\ref{sec:enrich-categ-funct}. One easily verifies that
  \begin{displaymath}
    (\cb{\Gamma}F)_{*}=\db{\Gamma}(F_{*})
    \quad\text{and}\quad
    (\cb{\Gamma} F)^{*}=\db{\Gamma}(F^{*}),
  \end{displaymath}
  for every \(\mcatW\)-functor \(F \colon \mathcal{A}\to \mathcal{B}\).
\end{remark}

Given also \(\Psi \colon \mathcal{B}\modto \mathcal{C}\), their \emph{composite} \(\Psi\cdot\Phi \colon \mathcal{A}\modto \mathcal{C}\) is defined by
\begin{displaymath}
  \Psi\cdot\Phi(A,C)=\int^{B\in\ob \mathcal{B}}\Psi(B,C)\otimes\Phi(A,B),
\end{displaymath}
whenever this coend exists. The coend above certainly exists if \(\mathcal{B}\) is small. But for any \(\mcatW\)-functors \(F \colon \mathcal{A}\to\mathcal{B}\) and \(G \colon \mathcal{C}\to \mathcal{B}\) one easily verifies that the following coends exist and are computed as
\begin{displaymath}
  \Psi\cdot F_{*}(A,C)
  =\int^{B\in\ob \mathcal{B}}\Psi(B,C)\otimes F_{*}(A,B)\cong \Psi(FA,C)
\end{displaymath}
and
\begin{displaymath}
  G^{*}\cdot\Psi(A,C)
  =\int^{B\in\ob \mathcal{B}}G^{*}(B,C)\otimes\Phi(A,B)\cong\Phi(A,GC).
\end{displaymath}
We obtain the bicategory
\begin{displaymath}
  \Dists{\mcatW}
\end{displaymath}
of small \(\mcatW\)-categories, \(\mcatW\)-distributors and \(\mcatW\)-natural transformations. The identity arrow at a (small) \(\mcatW\)-category \(\mathcal{A}\) is given by \(\mathcal{A}\colon \mathcal{A}\modto \mathcal{A}\).

\begin{example}
  For a quantale \(\mcatW=\mcatV\) and \(\mcatV\)-distributors \(\varphi \colon A\modto B\) and \(\psi \colon B\modto C\), their composite \(\psi\cdot\varphi \colon A\modto C\) is given by
  \begin{displaymath}
    (\psi\cdot\varphi)(a,c)=\bigvee\{\psi(b,c)\otimes\varphi(a,b)\mid b\in B\},
  \end{displaymath}
  for all \(a\in A\) and \(c\in C\).
\end{example}

\begin{example}
  For \(\mcatW=\Set\) and for distributors \(\Phi \colon\catA\modto\catB\) and \(\Psi \colon\catB\modto\catC\) between small categories, their composite \(\Psi\cdot\Phi \colon\catA\modto\catC\) is given by
  \begin{displaymath}
    (\Psi\cdot\Phi)(a,c)
    =\quot{\sum_{B\in\ob\catB}(\Psi(b,c)\times\Phi(a,b))}{\sim},
  \end{displaymath}
  where the equivalence relation \(\sim\) on \(\sum_{B\in\ob\catB}(\Psi(b,c)\times\Phi(a,b))\) is generated by all pairs
  \begin{displaymath}
    (v,\Phi(1_A,h)(u))\sim(\Psi(h,1_C)(v),u)
  \end{displaymath}
  with \(h \colon b\to b'\), \(v\in\Psi(b',c)\) and \(u\in\Phi(a,b)\). For more information we refer to \citep[Section~7.8]{Bor94}.
\end{example}

For a small \(\mcatW\)-category \(\mathcal{A}\), we may identify the category \([\mathcal{A},\mcatW]_\circ\) with the category of \(\mcatW\)-distributors of type \(\mathcal{E}\modto\mathcal{A}\) and of \(\mcatW\)-natural transformations between them; dually, \([\mathcal{A}^{\op},\mcatW]_\circ\) may be identified with the category of \(\mcatW\)-distributors of type \(\mathcal{A}\modto\mathcal{E}\) and of \(\mcatW\)-natural transformations between them. Furthermore, to every \(\mcatW\)-functor \(\Phi \colon \mathcal{A}\to \mcatW\) one associates the \(\mcatW\)-functor \(\Phi^{\vee}\colon \mathcal{A}^{\op}\to \mcatW\), called the \emph{Isbell conjugate} of \(\Phi\) (introduced by \citet{Isb60}, for more information see also \citep{AL21, Bae23}), defined by
\begin{displaymath}
  \Phi^{\vee}(A)
  =[\mathcal{A},\mcatW](\Phi,\mathcal{A}(A,-))
  =\int_{B\in\ob\mathcal{A}}[\Phi(B),\mathcal{A}(A,B)].
\end{displaymath}
Dually, the Isbell conjugate of \(\Psi \colon \mathcal{A}^{\op}\to \mcatW\) is the \(\mcatW\)-functor \(\Phi^{\vee}\colon \mathcal{A}\to \mcatW\) defined by
\begin{displaymath}
  \Psi^{\vee}
  =[\mathcal{A}^{\op},\mcatW](\Psi,\mathcal{A}(-,A))
  =\int_{B\in\ob\mathcal{A}}[\Psi(B),\mathcal{A}(B,A)].
\end{displaymath}
These constructions define a \(\mcatW\)-enriched adjunction
\begin{displaymath}
  \begin{tikzcd}[column sep=huge]
    {[\mathcal{A},\mcatW]^{\op}} %
    \ar[shift left, start anchor=east, end anchor=west, bend left=25, ""{name=U,below}]%
    {r}{(-)^{\vee}} %
    & {[\mathcal{A}^{\op},\mcatW]} %
    \ar[start anchor=west,end anchor=east,shift left,bend left=25, ""{name=D,above}] %
    {l}{(-)^{\vee}}
    \ar[from=U,to=D,symbol=\vdash]
  \end{tikzcd}
\end{displaymath}
since, for all \(\Phi \colon \mathcal{A}\to \mcatW\) and \(\Psi \colon \mathcal{A}^{\op}\to \mcatW\), we observe that
\begin{displaymath}
  [\mathcal{A}^{\op},\mcatW](\Psi,\Phi^{\vee})
  \cong
  [\mathcal{A}^{\op}\otimes \mathcal{A},\mcatW]
  (\Phi\cdot\Psi,\mathcal{A})
  \cong
  [\mathcal{A},\mcatW](\Phi,\Psi^{\vee}),
\end{displaymath}
naturally in \(\Phi\) and \(\Psi\). For instance, to see the first isomorphism, we calculate
\begin{align*}
  [\mathcal{A}^{\op},\mcatW](\Psi,\Phi^{\vee})
  &\cong\int_{A\in\ob \mathcal{A}^{\op}}
    \left[\Psi(A),\int_{B\in\ob \mathcal{A}}
    [\Phi(B),\mathcal{A}(A,B)]\right]\\
  &\cong\int_{A\in\ob \mathcal{A}^{\op}}\int_{B\in\ob \mathcal{A}}
    [\Psi(A)\otimes\Phi(B),\mathcal{A}(A,B)]\\
  &\cong\int_{(A,B)\in\ob(\mathcal{A}^{\op}\otimes \mathcal{A})}
    [\Phi(B)\otimes\Psi(A),\mathcal{A}(A,B)]\\
  &\cong
    [\mathcal{A}^{\op}\otimes\mathcal{A},\mcatW]
    (\Phi\cdot\Psi,\mathcal{A}).
\end{align*}
In particular, considering \(\Psi=\Phi^{\vee}\) and the unit \(E\to[\mathcal{A}^{\op},\mcatW](\Phi^{\vee},\Phi^{\vee})\), one obtains a \(\mcatW\)-natural transformation \(\varepsilon_{\Phi} \colon\Phi\cdot\Phi^{\vee}\to \mathcal{A}\) which exhibits \(\Phi^{\vee}\) as a right Kan lift of \(\mathcal{A}\colon \mathcal{A}\modto \mathcal{A}\) through \(\Phi \colon \mathcal{E}\modto \mathcal{A}\) in \(\Dists{\mcatW}\)
\begin{displaymath}
  \begin{tikzcd}
    \mathcal{A} &&
    \mathcal{A}\ar[distrib,""{name=D,above}]{ll}[swap]{\mathcal{A}}
    \ar[distrib,dotted]{dl}{\Phi^{\vee}}\\
    & \mathcal{E}\ar[distrib]{ul}{\Phi}
    \ar[to=D,shorten <=4pt,shorten >=6pt,Rightarrow]{}{\varepsilon_{\Phi}}
  \end{tikzcd}
\end{displaymath}
since the underlying equivalence functor
\begin{displaymath}
  [\mathcal{A}^{\op},\mcatW](\Psi,\Phi^{\vee})_{\circ}
  \longrightarrow
  [\mathcal{A}^{\op}\otimes \mathcal{A},\mcatW]_{\circ}
\end{displaymath}
sends \(\alpha \colon \Psi\to\Phi^{\vee}\) to \(\varepsilon_{\Phi} (\Phi\alpha)\).

\begin{proposition}
  Let \(\mathcal{A}\) be a small \(\mcatW\)-category. A \(\mcatW\)-distributor \(\Phi \colon \mathcal{E}\modto \mathcal{A}\) is left adjoint in \(\Dists{\mcatW}\) if and only if \(\Phi\dashv\Phi^{\vee}\) in \(\Dists{\mcatW}\), with counit \(\varepsilon_{\Phi}\). Dually, a \(\mcatW\)-distributor \(\Psi \colon \mathcal{A}\modto \mathcal{E}\) is right adjoint in \(\Dists{\mcatW}\) if and only if \(\Psi^{\vee}\dashv\Psi\), with the counit defined similarly.
\end{proposition}
\begin{proof}
  See \citep[Proposition~10.4]{AL21}.
\end{proof}

\section{Completeness \emph{à la} Lawvere}
\label{sec:compl-la-lawv}

Throughout this section \(\mcatW\) denotes a closed symmetric monoidal complete and cocomplete category.

\subsection{Lawvere completeness for enriched categories}
\label{sec:general-notion}

We start by recalling the well-known key property of the \(\mcatW\)-distributors induced by a \(\mcatW\)-functor.

\begin{proposition}
  \label{d:prop:3}
  For every \(\mcatW\)-functor \(F \colon \mathcal{A}\to \mathcal{B}\) between small \(\mcatW\)-categories, \(F_{*}\dashv F^{*}\) in \(\Dists{\mcatW}\).
\end{proposition}
The unit and the counit of \(F_{*}\dashv F^{*}\) have components
\begin{displaymath}
  \mathcal{A}(A,A')\longrightarrow \mathcal{B}(FA,FA')\cong F^{*}\cdot F_{*}(A,A')
\end{displaymath}
and
\begin{align*}
  F_{*}\cdot F^{*}(B,B')
  &\cong\int^{A\in\ob \mathcal{A}}F_{*}(A,B')\otimes F^{*}(B,A)\\
  &\cong\int^{A\in\ob \mathcal{A}}\mathcal{B}(FA,B')\otimes \mathcal{B}(B,FA)
  \longrightarrow \mathcal{B}(B,B'),
\end{align*}
respectively.
\begin{definition}
  A small \(\mcatW\)-category \(\mathcal{A}\) is called \emph{Lawvere complete} whenever every left adjoint \(\mcatW\)-distributor \(\Phi \colon \mathcal{X}\modto \mathcal{A}\) in \(\Dists{\mcatW}\) is representable, that is, whenever there exists a \(\mcatW\)-functor \(F \colon \mathcal{X}\to \mathcal{A}\) so that \(\Phi\cong F_{*}\).
\end{definition}

We note that, if \(\Phi\cong F_{*}\) and \(\Phi\dashv\Psi\) in \(\Dists{\mcatW}\), then also \(\Psi\cong F^{*}\) by Proposition~\ref{d:prop:3}.

\begin{remark}
  It suffices to consider left adjoints \(\Phi \colon\mathcal{E}\modto \mathcal{A}\) in the definition above. Also recall that an adjunction \(\Phi\dashv\Psi\) in \(\Dists{\mcatW}\) with \(\Psi \colon \mathcal{A}\modto \mathcal{E}\) is given by units \(\eta\) and \(\varepsilon\)
  \begin{align*}
    E=\mathcal{E}(\star,\star)
    &\xrightarrow{\quad\eta=\eta_{\star}\quad}
      \int^{A\in\ob \mathcal{A}}\Psi(A)\otimes\Phi(A)
      =\Psi\cdot\Phi(\star,\star),\\
    \mathcal{A}(A,B)
    &\xleftarrow{\quad\varepsilon_{A,B}\quad}
      \Phi(B)\otimes \Psi(A)=\Phi\cdot\Psi(A,B),
  \end{align*}
  subject to the usual equations. In particular, for \(\mcatW=\mcatV\) a quantale and \((\varphi \colon E\modto A)\dashv(\psi \colon A\modto E)\) in \(\Dists{\mcatV}\), the unit and the counit are given by the inequalities
  \begin{equation}
    \label{d:eq:9}%
    \kk\leq\bigvee\{\psi(c)\otimes\varphi(c)\mid c\in A\}
    \quad\text{and}\quad
    \varphi(b)\otimes\psi(a)\leq X(a,b)\qquad(a,b\in A),
  \end{equation}
  respectively. Also note that we may identify any \(\mcatW\)-functor \(\mathcal{E}\to \mathcal{A}\) with the image \(A\) of the unique object of \(\mathcal{E}\), and we write \(A_{*}\colon \mathcal{E}\modto \mathcal{A}\) and \(A^{*}\colon \mathcal{A}\modto \mathcal{E}\) for the induced \(\mcatW\)-distributors. The left adjoint \(\Phi \colon \mathcal{E}\modto \mathcal{A}\) (and the adjunction \(\Phi\dashv\Psi\)) is representable if and only if there exists an object \(A\) in \(\mathcal{A}\) with \(A_{*}\cong\Phi\) (and necessarily \(A^{*}\cong\Psi\)).
\end{remark}

\begin{remark}\label{w:rem:1}
  The notion of Lawvere completeness is self dual in the sense that \(\mathcal{A}\) is Lawvere complete if and only if \(\mathcal{A}^{\op}\) is Lawvere complete. One way to see this is to observe that, by Remark~\ref{d:rem:3}, the passage \(\mathcal{A}\mapsto \mathcal{A}^{\op}\) extends to a pseudofunctor \((-)^{\op}\colon\Dists{\mcatW}^{\op}\to\Dists{\mcatW}\) which sends \(\Phi \colon \mathcal{A}\modto \mathcal{B}\) to \(\Phi^{\op}\colon \mathcal{B}^{\op}\modto \mathcal{A}^{\op}\) defined as the composite
  \begin{displaymath}
    (\mathcal{B}^{\op})^{\op}\otimes \mathcal{}A^{\op}
    =\mathcal{B}\otimes \mathcal{A}^{\op}
    \longrightarrow \mathcal{A}^{\op}\otimes \mathcal{B}
    \xrightarrow{\quad\Phi\quad}\mcatW,
  \end{displaymath}
  and that \((\Phi^{\op})^{\op}=\Phi\). The pseudofunctor \((-)^{\op}\) sends the adjunction \((\Phi \colon \mathcal{E}\modto \mathcal{A})\dashv(\Psi \colon \mathcal{A}\modto \mathcal{E})\) to the adjunction \((\Psi^{\op}\colon \mathcal{E}= \mathcal{E}^{\op}\modto \mathcal{A}^{\op})\dashv(\Phi^{\op} \colon \mathcal{A}^{\op}\modto \mathcal{E})\), hence left adjoint distributors \(\mathcal{E}\modto \mathcal{A}\) are in bijection with left adjoint distributors \(\mathcal{E}\modto \mathcal{A}^{\op}\). Furthermore, for every \(\mcatW\)-functor \(F \colon \mathcal{A}\to \mathcal{B}\), \((F_{*})^{\op}= (F^{\op})^{*}\) and \((F^{*})^{\op}=(F^{\op})_{*}\). Therefore every left adjoint distributor \(\mathcal{E}\modto \mathcal{A}\) is representable if and only if every left adjoint distributor \(\mathcal{E}\modto \mathcal{A}^{\op}\) is representable.
\end{remark}

\begin{remark}\label{d:rem:2}
  Alternatively, Lawvere completeness of a small \(\mcatW\)-category \(\mathcal{A}\) may be equivalently described as cocompleteness with respect to the class of all right adjoint weights \(\Psi \colon \mathcal{A}\modto \mathcal{E}\) (see \citep{Str83, BD86}). With this description as the definition of Lawvere completeness for a general (not necessarily small) \(\mcatW\)-category \(\mathcal{X}\), we conclude that every cocomplete \(\mcatW\)-category \(\mathcal{X}\) is Lawvere complete; hence, in particular, the \(\mcatW\)-category \(\mcatW\) is Lawvere complete.
\end{remark}

\subsection{Quantale-enriched categories}
\label{sec:quant-enrich-categ}

\begin{proposition}
  \label{d:prop:4}%
  Let \(\mcatV\) be a quantale. A left adjoint \(\mcatV\)-distributor \(\varphi\colon E\modto A\) with right adjoint \(\psi \colon A\modto E\) is representable if and only if there exists an element \(a\in A\) with \(\kk\leq\varphi(a)\) and \(\kk\leq\psi(a)\).
\end{proposition}
\begin{proof}
  A representing element \(a\in A\) for \(\varphi\dashv\psi\) must clearly satisfy \(\kk\leq\varphi(a)\) and \(\kk\leq\psi(a)\). Conversely, by the Yoneda lemma, the conditions \(\kk\leq\varphi(a)\) and \(\kk\leq\psi(a)\) imply \(a_{*}\leq\varphi\) and \(a^{*}\leq\psi\). With the Isbell adjunction one obtains \(a_{*}={a^{*}}^{\vee}\geq\psi^{\vee}=\varphi\), hence \(\varphi=a_{*}\).
\end{proof}

\begin{example}
  Consider the quantale \(\mcatV=\two\) and an adjunction \((\varphi \colon E\modto A)\dashv(\psi \colon A\modto E)\). The first inequality of \eqref{d:eq:9} tells us that there exists \(c\in A\) so that \(1=\varphi(c)\) and \(1=\psi(c)\), that is, that \(\varphi \colon E\modto A\) is representable. Since \(\Cats{\mcatV}=\Ord\), it follows that every ordered set is Lawvere complete. More generally, if the unit \(\kk\) is \emph{totally compact} in any quantale \(\mcatV\) (meaning that \(\kk\leq\bigvee S\) implies \(\kk\leq s\) for some \(s\in S\), for all \(S\subseteq\mcatV\)) and satisfies \(k\leq u\otimes v\) only if \(\kk\leq u\) and \(\kk\leq v\), for all \(u,v\in\mcatV\), then every \(\mcatV\)-category is Lawvere complete.
\end{example}

\begin{example}
  A set, viewed as a \(\one\)-category, is Lawvere complete if and only if it is non-empty.
\end{example}

\begin{example}
  We consider now the quantale \(\mcatV= [0,\infty]\) with addition and unit element \(0\), ordered by \(\geq\), so that \(\Cats{\mcatV}=\Met\) (see Example~\ref{d:ex:4}). As observed in \citep{Law73}, a metric space is Lawvere complete if and only if it is Cauchy complete in the classic sense that every Cauchy sequence converges.

  As usual, for a (generalised) metric space \(X\) we consider the induced (symmetric) topology where, for \(x\in X\) and \(A\subseteq X\),
  \begin{displaymath}
    x\in\overline{A}\iff 0=\inf\{X(x,a)+X(a,x)\mid a\in A\}.
  \end{displaymath}
  Note that \(X\) and \(X^{\op}\) induce the same topology on the set \(X\), and the topology induced by \([0,\infty]\) is the usual Euclidean topology: for \(u\in [0,\infty]\) and \(S\subseteq [0,\infty]\),
  \begin{displaymath}
    u\in\overline{S}
    \iff 0=\inf\{|v-u|\mid v\in S\}.
  \end{displaymath}
  For the proof, assume first that \(X\) is Cauchy complete. Let \((\varphi \colon E\modto X)\dashv(\psi \colon X\modto E)\) be an adjunction in \(\Dists{[0,\infty]}\). We note that both maps \(\varphi \colon X\to [0,\infty]\) and \(\psi \colon X\to [0,\infty]\) are continuous, therefore also the map
  \begin{displaymath}
    \sigma \colon X \longrightarrow [0,\infty],\quad
    x \longmapsto \varphi(x)+\psi(x)
  \end{displaymath}
  is continuous. By \eqref{d:eq:9}, for every \(n\in\mathbb{N}\) we can choose some \(x_n\in X\) so that \(\sigma(x_n)=\varphi(x_n)+\psi(x_n)\leq\frac{1}{n}\), and the sequence \((x_n)_{n\in\mathbb{N}}\) is Cauchy since
  \begin{displaymath}
    X(x_n,x_m)\leq \varphi(x_m)+\psi(x_n)
    \leq \sigma(x_m)+\sigma(x_n)
    \leq \frac{1}{m}+ \frac{1}{n},
  \end{displaymath}
  for all \(n,m\in \mathbb{N}\). Since \(X\) is Cauchy complete, \((x_n)_{n\in\mathbb{N}}\) converges to some \(x\in X\), therefore \((\sigma(x_n))_{n\in\mathbb{N}}\) converges to \(\sigma(x)\), hence \(0=\sigma(x)=\varphi(x)+\psi(x)\). Consequently, \(\varphi(x)=0\) and \(\psi(x)=0\), by Proposition~\ref{d:prop:4}, the adjunction \(\varphi\dashv\psi\) is representable.

  Conversely, assume now that \(X\) is Lawvere complete. We recall from \citep{HT10} that, for the Yoneda embedding \(y_X \colon X\to [X^{\op},[0,\infty]]\),
  \begin{displaymath}
    \overline{y_X(X)}
    =\{\psi\in [X^{\op},[0,\infty]]
    \mid\psi \colon X\modto E\text{ is right adjoint}\}.
  \end{displaymath}
  Since \(X\) is Lawvere complete, \(y_X(X)=\overline{y_X(X)}\) is a closed subset of the Cauchy complete space \([X^{\op},[0,\infty]]\), therefore \(y_X(X)\) is Cauchy complete and, since \(X\cong y_X(X)\), also \(X\) is Cauchy complete.
\end{example}

\subsection{Lawvere completeness for ordinary categories}
\label{sec:ordinary-categories}

If the symmetric monoidal closed category \(\mcatW\) is not thin, then the relation between adjoint distributors and representable distributors is more subtle. We start by recalling the case \(\mcatW=\Set\), see \citep{Bor94, BD86} for more details. Firstly, for \((\Phi\colon\catE\modto\catA)\dashv(\Psi\colon\catA\modto\catE)\) in \(\Dist\), the unit
\begin{displaymath}
  \eta \colon \catE \longrightarrow
  \left(\quot{\sum_{a\in\ob\catA}(\Psi(a)\times\Phi(a))}{\sim}\right)
\end{displaymath}
of this adjunction is given by an equivalence class \([(v,u)]\), for some object \(a\) in \(\catA\) and elements \(u\in\Phi(a)\) and \(v\in\Psi(a)\), and we write then \(\eta=[(v,u)]\). By the proof of \citep[Proposition~7.9.2]{Bor94}, we have the following
\begin{lemma}\label{d:lem:3}
  For a small category \(\catA\) and distributors \(\Phi\colon\catE\modto\catA\) and \(\Psi\colon\catA\modto\catE\), an adjunction \(\Phi\dashv\Psi\) in \(\Dist\) is given by
  \begin{itemize}
  \item a family of maps \(\varepsilon_{a,b}\colon\Phi(b)\times\Psi(a)\to\catA(a,b)\), natural in \(a\) and \(b\),
  \item an object \(c\) of \(\catA\) and elements \(u\in\Phi(c)\) and \(v\in\Psi(c)\)
  \end{itemize}
  so that, for all \(a,b\in\ob\catA\), \(y\in\Phi(b)\) and \(x\in\Psi(a)\),
  \begin{displaymath}
    x=\Psi(\varepsilon_{a,c}(u,x))(v)
    \quad\text{and}\quad
    y=\Phi(\varepsilon_{c,b}(y,v))(u).
  \end{displaymath}
  Then the unit \(\eta\) is given by \(\eta=[(v,u)]\) and the counit \(\varepsilon\) is defined by the family \((\varepsilon_{a,b})_{a,b\in\ob\catA}\).
\end{lemma}

Furthermore, also the following result is essentially contained in the proof of \citep[Proposition~7.9.2]{Bor94}. Here we shorten the proof using the Isbell conjugation.

\begin{lemma}
  \label{d:lem:1}
  Let \(\Phi \colon\catE\modto\catA\) in \(\Dist\). Consider an object \(a\) of \(\catA\) and natural transformations
  \begin{equation}
    \label{d:eq:6}
    \alpha \colon\catA(a,-)\longrightarrow\Phi
    \quad\text{and}\quad
    \beta \colon\Phi \longrightarrow\catA(a,-).
  \end{equation}
  By the Yoneda Lemma, \(\alpha\) corresponds to an element \(u\in\Phi(a)\) and \(\beta^{\vee}\colon\catA(-,a)\to\Phi^{\vee}\) to \(v=\beta\in\Phi^{\vee}(a)\). Then the following assertions are equivalent:
  \begin{tfae}
  \item\label{d:item:1} \(\Phi\dashv\Phi^{\vee}\) with unit \(\eta=[(v,u)]\),
  \item\label{d:item:2} \(\alpha\beta=1_{\Phi}\).
  \end{tfae}
\end{lemma}
\begin{proof}
  Assume first \ref{d:item:1}. Then \(\alpha\) and \(\beta\) correspond (up to an interchange of variables) to the natural transformations \(\delta\) and \(\gamma\) of \citep[Page~317]{Bor94}, hence \(\alpha\beta=1_{\Phi}\). Assume now \ref{d:item:2}. Let us recall first the argument in the proof of \citep[Proposition~7.9.2]{Bor94}: \(\beta\alpha \colon\catA(a,-)\to\catA(a,-)\) is idempotent and, by the Yoneda Lemma, it corresponds to an idempotent \(e \colon a\to a\) in \(\catA\), which in turn corresponds to an idempotent natural transformation \(\catA(-,e)\colon\catA(-,a)\to\catA(-,a)\). Furthermore, \(\catA(-,e)\) splits in \([\catA^{\op},\Set]\):
  \begin{displaymath}
    \begin{tikzcd}
      \catA(-,a)\ar{rr}{\catA(-,e)}\ar{dr}[swap]{\gamma}
      && \catA(-,a).\\
      & \Psi\ar{ur}[swap]{\delta}
    \end{tikzcd}
  \end{displaymath}
  Then \(\Psi\) is shown to be right adjoint to \(\Phi\), with unit \(\eta=[(\gamma_a(1_{a}),\alpha_a(1_a))]\). To conclude \ref{d:item:1}, just observe that \(\catA(-,e)\) is isomorphic to \((\beta\alpha)^{\vee}=\alpha^{\vee}\beta^{\vee}\) (see \citep[Lemma~2.2]{AL21}), hence \(\eta=[(\beta^{\vee}_a(1_a),\alpha_a(1_a))]=[(v,u)]\).
\end{proof}
We say that \(\Phi \colon\catE\modto\catA\) in \(\Dist\) is \emph{a retract of a representable distributor} whenever there are an object \(a\) of \(\catA\) and natural transformations as in \eqref{d:eq:6} with \(\alpha\beta=1_{\Phi}\). Lemma~\ref{d:lem:1} implies immediately the following two results.

\begin{proposition}[{\citep[Proposition~7.9.2]{Bor94}}]
  The distributor \(\Phi \colon\catE\modto\catA\) is left adjoint in \(\Dist\) if and only if it is a retract of a representable distributor.
\end{proposition}

\begin{theorem}[{\citep[Proposition~7.9.3]{Bor94}}]
  A small category \(\catA\) is Lawvere complete if and only if every idempotent splits in \(\catA\).
\end{theorem}

\subsection{Lawvere completeness for \(\mcatV\)-normed categories}
\label{sec:normed-categories}

Our principal interest lies in the case \(\mcatW=\Setss{\mcatV}\), for an arbitrary quantale \(\mcatV\), which we analyse now. First note that, by Proposition~\ref{d:prop:1} and Remark~\ref{d:rem:1}, the strict monoidal functor \(\mO \colon\SetV\to\Set\) induces a strict pseudofunctor
\begin{displaymath}
\db{\mO} \colon\DistV \longrightarrow\Dist
\end{displaymath}
where \(\db{\mO}(\ncatX)=\cb{\mO}(\ncatX)\) for each \(\mcatV\)-normed category \(\ncatX\); moreover, \(\db{\mO}\) is faithful on 2-cells. Consequently, for \(\Phi \colon \ncatE\modto\ncatA\) and \(\Psi \colon\ncatA\modto\ncatE\) in \(\DistV\), \(\Phi\dashv\Psi\) if and only if \(\db{\mO}(\Phi)\dashv\db{\mO}(\Psi)\) in \(\Dist\) with units being normed natural transformations. This will allow us to transport several facts about distributors between ordinary categories to the \(\mcatV\)-normed context. In particular:
\begin{lemma}\label{d:lem:4}
  Let \(\Phi \colon \ncatE\modto\ncatA\) and \(\Psi \colon\ncatA\modto\ncatE\) be in \(\DistV\). Then an adjunction \(\Phi\dashv\Psi\) in \(\DistV\) is given by the same data as in Lemma~\ref{d:lem:3}, where additionally every \(\varepsilon_{a,b}\) is \(\mcatV\)-normed and \(\kk\leq|[(v,u)]|\).
\end{lemma}

Let us first analyse the notion of Lawvere completeness for \(\mcatV\)-normed categories coming from \(\mcatV\)-categories. For the symmetric strong monoidal functor \(\mi \colon \mcatV\to\SetV\), the correspondence of Remark~\ref{d:rem:1} induces a pseudofunctor
\begin{displaymath}
  \db{\mi} \colon\Dists{\mcatV} \longrightarrow\DistV
\end{displaymath}
which sends a \(\mcatV\)-distributor \(\varphi \colon X\modto Y\) to \(\Phi \colon\ncatX\modto\ncatY\) where \(\ncatX=\cb{\mi}(X)\), \(\ncatY=\cb{\mi}(Y)\) and \(\Phi(x,y)=(\star,\varphi(x,y))\). Furthermore, the symmetric strict monoidal functor \(\ms \colon\SetV\to\mcatV\) induces, for all \(\mcatV\)-normed categories \(\ncatA\) e \(\ncatB\), the functor
\begin{displaymath}
  \db{\ms} \colon \DistV(\ncatA,\ncatB)
  \longrightarrow\Dists{\mcatV}(\cb{\ms}(\ncatA),\cb{\ms}(\ncatB))
\end{displaymath}
sending \(\Phi \colon\ncatA\modto\ncatB\) to \(\varphi \colon \cb{\ms}(\ncatA)\modto\cb{\ms}(\ncatB)\) where \(\varphi(a,b)=\bigvee_{u\in\Phi(a,b)}|u|\). Clearly, for every \(\mcatV\)-distributor \(\varphi \colon X\modto Y\) we have \(\varphi=\db{\ms}\db{\mi}(\varphi)\), and for every \(\mcatV\)-normed distributor \(\Phi \colon\ncatA\modto\ncatB\), one has \(\Phi=\db{\mi}\db{\ms}(\Phi)\) if and only if \(\mO(\Phi(a,b))=\{\star\}\) for all \(a\in\ob\ncatA\) and all \(b\in\ob\ncatB\).

For a \(\mcatV\)-category \(X\), consider the \(\mcatV\)-normed category \(\ncatX=\cb{\mi}(X)\) and an adjunction \((\Phi \colon\ncatE\modto\ncatX)\dashv(\Psi \colon\ncatX\modto\catE)\) in \(\DistV\). Then also \(\db{\mO}(\Phi)\dashv\db{\mO}(\Psi)\) in \(\Dist\); therefore, by Lemma~\ref{d:lem:1}, there is an object \(x\) in \(\ncatX\) and split epimorphic natural transformations
\begin{displaymath}
  \db{\mO}(\ncatX)(x,-)\longrightarrow \db{\mO}(\Phi)
  \quad\text{and}\quad
  \db{\mO}(\ncatX)(-,x)\longrightarrow \db{\mO}(\Psi).
\end{displaymath}
Therefore, for every \(y\in\ob\ncatX\), both \(\db{\mO}(\Phi)(y)\) and \(\db{\mO}(\Psi)(y)\) are one-element sets. We have shown:
\begin{lemma}
  Let \(X\) be a \(\mcatV\)-category and consider the \(\mcatV\)-normed category \(\ncatX=\cb{\mi}(X)\). Then \(\mi \colon\mcatV\to\SetV\) induces an equivalence
  \begin{displaymath}
    \db{\mi}\colon
    \{\varphi \colon E\modto X\text{ left adjoint}\}
    \longrightarrow
    \{\Phi \colon \ncatE\modto\ncatX\text{ left adjoint}\}.
  \end{displaymath}
  Moreover, for every \(x\in X\), \(\db{\mi}(X(x,-))= (\cb{\mi}X)(x,-)\).
\end{lemma}

\begin{corollary}
  \label{d:cor:1}
  Let \(X\) be a \(\mcatV\)-category. Then \(X\) is Lawvere complete if and only if the \(\mcatV\)-normed category \(\cb{\mi}(X)\) is Lawvere complete.
\end{corollary}

A normed version of Lemma~\ref{d:lem:1} reads as follows.

\begin{lemma}\label{d:lem:2}
  Let \(\Phi \colon\ncatE\modto\ncatA\) in \(\DistV\). Consider an object \(a\) of \(\ncatA\) and (not necessarily \(\mcatV\)-normed) natural transformations
  \begin{equation}
    \label{d:eq:8}
    \alpha \colon\ncatA(a,-)\longrightarrow\Phi
    \quad\text{and}\quad
    \beta \colon\Phi \longrightarrow\ncatA(a,-);
  \end{equation}
  here \(\alpha\) corresponds to an element \(u\in\Phi(a)\) and \(\beta^{\vee}\colon\catA(-,a)\to\Phi^{\vee}\) to \(v=\beta\in\Phi^{\vee}(a)\). Then the following assertions are equivalent,
  \begin{tfae}
  \item\label{d:item:3} \(\Phi\dashv\Phi^{\vee}\) in \(\DistV\) with unit \(\eta=[(v,u)]\) where \(\kk\leq|u|\) and \(\kk\leq|v|\),
  \item\label{d:item:4} the natural transformations \(\alpha\) and \(\beta\) are \(\mcatV\)-normed and \(\alpha\beta=1_{\Phi}\) in \(\DistV\).
  \end{tfae}
\end{lemma}
\begin{proof}
  Assume first \ref{d:item:3}. Then \(\kk\leq |u|=|\alpha|\) and \(\kk\leq |\beta|\), that is, \(\alpha\) and \(\beta\) are \(\mcatV\)-normed natural transformations. Therefore, \ref{d:item:4} follows with Lemma~\ref{d:lem:1}. Assume now \ref{d:item:4}. Then, since \(\alpha\) and \(\beta\) are \(\mcatV\)-normed, \(\kk\leq |\alpha|=|u|\) and \(\kk\leq|\beta|=|v|\), therefore \(\kk\leq |[(v,u)]|\). Hence, the unit \(\eta\) is a \(\mcatV\)-normed natural transformation. Following the proof of \citep[Proposition~7.9.2]{Bor94}, the counit \(\varepsilon \colon \Phi\cdot\Phi^{\vee}\to\ncatA\) has components
  \begin{align*}
    \varepsilon_{b,c}\colon \Phi(c)\otimes\Phi^{\vee}(b)
    & \longrightarrow \ncatA(b,c).\\
    (w_1,w_2)
    &\longmapsto \beta_c(w_1)\cdot \alpha^{\vee}_b(w_2)
  \end{align*}
  Since \(\kk\leq|\alpha|\leq|\alpha^{\vee}|\), also \(\alpha^{\vee}\) is \(\mcatV\)-normed, and we compute
  \begin{displaymath}
    |\beta_c(w_1)\cdot \alpha^{\vee}_b(w_2)|
    \geq |\beta_c(w_1)|\otimes |\alpha^{\vee}_b(w_2)|
    \geq |w_1|\otimes|w_2|=|(w_1,w_2)|.
  \end{displaymath}
  We conclude that also \(\varepsilon\colon \Phi\cdot\Phi^{\vee}\to\ncatA\) is \(\mcatV\)-normed, which proves~\ref{d:item:3}.
\end{proof}

\begin{definition}\label{w:def:1}
  We say that \(\Phi \colon\ncatE\modto\ncatA\) in \(\DistV\) is a \emph{\(\mcatV\)-normed retract of a representable distributor} whenever there is an object \(a\) of \(\ncatA\) and \(\mcatV\)-normed natural transformations
  \begin{displaymath}
    \alpha \colon\ncatA(a,-)\longrightarrow\Phi
    \quad\text{and}\quad
    \beta \colon\Phi \longrightarrow\ncatA(a,-),
  \end{displaymath}
  with \(\alpha\beta=1_{\Phi}\).

  We say that a left adjoint \(\mcatV\)-normed distributor \(\Phi \colon\ncatE\modto\ncatA\) \emph{has a presentable unit} whenever there exist an object \(a\) in \(\ncatA\) and elements \(u\in\Phi(a)\) and \(v\in\Phi^{\vee}(a)\) so that the unit \(\eta\) of the adjunction \(\Phi\dashv\Phi^{\vee}\) is of the form \(\eta=[(v,u)]\) and \(\kk\leq|u|\) and \(\kk\leq|v|\).
\end{definition}

By the uniqueness of adjoints in bicategories, given adjunctions \((\Phi\colon\ncatE\modto\ncatA)\dashv(\Psi \colon\ncatA\modto\ncatE)\) and \((\Phi\colon\ncatE\modto\ncatA)\dashv(\Psi' \colon\ncatA\modto\ncatE)\) in \(\DistV\) with units \(\eta\) and \(\eta'\), respectively, and given an object \(a\) in \(\ncatA\) and an element \(u\in\Phi(a)\) with \(\kk\leq|u|\), then there exists \(v\in\Psi(a)\) with \(\kk\leq|v|\) and \(\eta=[(v,u)]\) if and only if there exists \(v'\in\Psi'(a)\) with \(\kk\leq|v'|\) and \(\eta'=[(v',u)]\). In particular, this notion does not depend on the choice of the unit; and similarly, with \(\Phi\) also every distributor isomorphic to \(\Phi\) has a presentable unit. The adjunction \(\ncatA(a,-)\dashv\ncatA(-,a)\) has the unit \(\eta=[(1_a,1_a)]\); therefore, every representable left adjoint distributor has a presentable unit.

\begin{proposition}
  Let \(\Phi \colon\ncatE\modto\ncatA\) be in \(\DistV\). Then the following assertions are equivalent.
  \begin{tfae}
  \item \(\Phi\) is a \(\mcatV\)-normed retract of a representable distributor.
  \item \(\Phi\) is left adjoint in \(\DistV\) and has a presentable unit.
  \end{tfae}
\end{proposition}
\begin{proof}
  Follows immediately from Lemma~\ref{d:lem:2}, after the observation that \(|\alpha|=|u|\) (by the normed Yoneda Lemma) and \(|\beta|=|v|\).
\end{proof}

\begin{example}
  For \(\mcatV=\one\), we are in the context of Subsection~\ref{sec:ordinary-categories} since \(\DistV\cong\Dist\) in this case. Clearly, every left adjoint \(\Phi \colon\catE\modto\catA\) in \(\Dist\) has a presentable unit.
\end{example}

\begin{example}
  Similarly, for \(\mcatV=\two\), every left adjoint \(\Phi \colon\ncatE\modto\ncatA\) in \(\DistV\) has a presentable unit. In fact, this is true for every quantale \(\mcatV\) where the unit \(\kk\) has the property that, for every non-empty subset \(A\subseteq\mcatV\), \(\kk\leq\bigvee A\) implies that \(\kk\leq a\) for some \(a\in A\).
\end{example}

\begin{example}
  \label{d:ex:2}
  Consider now a quantale \(\mcatV\) and a left adjoint \(\mcatV\)-distributor \(\varphi \colon E\modto X\) with right adjoint \(\psi \colon X\modto E\). Then \(\db{\mi}(\varphi)\) in \(\DistV\) has a presentable unit if and only if there is some \(x\in X\) with \(\kk\leq\varphi(x)\) and \(\kk\leq\psi(x)\).
\end{example}

Using the functor
\begin{displaymath}
  (-)_{\circ}\colon\DistV(\ncatE,\ncatA)
  \longrightarrow
  \Dist(\ncatE_{\circ}=\catE,\ncatA_{\circ}),
\end{displaymath}
we see that, if \(\Phi\) is a \(\mcatV\)-normed retract of \(\ncatA(a,-)\), then \(\Phi_{\circ}\) is a retract of \(\ncatA(a,-)_{\circ}=\ncatA_{\circ}(a,-)\), hence \(\Phi_{\circ}\) is a left adjoint distributor. In fact, a little more can be said:

\begin{proposition}
  If \(\Phi \colon\ncatE\modto\ncatA\) in \(\DistV\) is left adjoint with presentable unit and right adjoint \(\Psi \colon\ncatA\modto\ncatE\). Then \(\Phi_{\circ}\dashv\Psi_{\circ}\) in \(\Dist\).
\end{proposition}
\begin{proof}
  Combine Lemmas~\ref{d:lem:3} and \ref{d:lem:4}.
\end{proof}

\begin{corollary}
  If \(\Phi \colon\ncatE\modto\ncatA\) in \(\DistV\) is left adjoint with presentable unit, then \((\Phi^{\vee})_{\circ}\cong\Phi_{\circ}^{\vee}\).
\end{corollary}

\begin{theorem}\label{d:thm:2}
  A small \(\mcatV\)-normed category \(\ncatA\) is Lawvere complete if and only if
  \begin{enumerate}
  \item the ordinary category \(\ncatA_{\circ}\) is idempotent complete;
  \item every left adjoint \(\mcatV\)-distributor \(\Phi\colon\ncatE\modto\ncatA\) has a presentable unit.
  \end{enumerate}
\end{theorem}
\begin{proof}
  Assume first that \(\ncatA\) is Lawvere complete. Then every left adjoint \(\mcatV\)-normed distributor \(\Phi \colon\ncatE\to\ncatA\) is of the form \(\Phi\cong\ncatA(a,-)\), for some object \(a\) of \(\ncatA\), and therefore has a presentable unit. Furthermore, if \(e \colon a\to a\) is an idempotent in \(\ncatA\) with \(\kk\leq |e|\), then the \(\kk\)-idempotent \(\ncatA(e,-)\colon\ncatA(a,-)\to\ncatA(a,-)\) in \([\ncatA,(\nSetV)]\) splits with \(\kk\)-morphisms \(\alpha \colon\ncatA(a,-)\to\Phi\) and \(\beta \colon\Phi\to\ncatA(a,-)\) where \(\beta\alpha=1_{\Phi}\) since \([\ncatA,(\nSetV)]_{\circ}\) is complete and therefore idempotent complete (see \citep[Sections~3.3 and 3.8]{Kel82}). By Lemma~\ref{d:lem:1}, \(\Phi\) is left adjoint in \(\DistV\) and therefore is of the form \(\Phi\cong\ncatA(b,-)\) for some \(b\in\ob\ncatA\), by the Yoneda Lemma, this provides a splitting of \(e\).

  Assume now that \(\ncatA\) satisfies the two conditions. Then every left adjoint \(\mcatV\)-normed distributor \(\Phi \colon\ncatE\modto\ncatA\) is a \(\mcatV\)-normed retract of a representable distributor and, since \(\ncatA_{\circ}\) is idempotent complete, \(\Phi\) is also representable.
\end{proof}

\begin{example}
  For \(\mcatV=\one\) and for \(\mcatV=\two\), a \(\mcatV\)-normed small category \(\ncatA\) is Lawvere complete if and only if \(\ncatA_{\circ}\) is idempotent complete.
\end{example}

\begin{example}
  Theorem~\ref{d:thm:2} and Example~\ref{d:ex:2} imply also the affirmation of Corollary~\ref{d:cor:1}.
\end{example}

\section{Completeness \emph{à la} Cauchy}
\label{sec:compl-la-cauchy}

\subsection{Normed convergence}
\label{sec:normed-convergence}

In the previous section we have seen how the classic metric notion of ``Cauchy completeness'' relates in a metric space to the categorical question of representability of distributors. We stress that this notion is intrinsically symmetric: a sequence \(s\) in a metric space \(X\) is Cauchy if and only if it is Cauchy in \(X\) if and only if it is Cauchy in the symmetrisation of \(X\), and similarly, \(s\) has a limit in \(X\) if and only if it has a limit in \(X^{\op}\) if and only if it has a limit in the symmetrisation of \(X\). A non-symmetric notion of ``Cauchy completeness'' was introduced by \citet{BBR98}: a sequence $s=(x_n)$ in a metric space \(X\) is \emph{(forward) Cauchy} whenever
\begin{displaymath}
  \inf_{N\in\ncatN}\sup_{n\geq m\geq N}X(x_m,x_n)=0,
\end{displaymath}
and an element \(x\in X\) is a \emph{(forward) limit} of $s$ whenever
\begin{displaymath}
  X(x,y)=\inf_{N\in\ncatN}\sup_{n\geq N}X(x_n,y),
\end{displaymath}
for all $y\in X$. Then a metric space \(X\) is \emph{complete} if every Cauchy sequence in \(X\) has a limit.

For normed categories, (limits of) Cauchy sequences of arrows were considered in \citep{Kub17, Nee20}. A revised notion and theory of normed convergence in the context of \(\mcatV\)-normed categories was developed in \citep{CHT25}, which we recall first. We consider here the ordered set $\ncatN$ as a category, discretely $\mcatV$-normed with constant value $\bot$ for all non-identical arrows. This way a $\mcatV$-normed functor $\ncatN\to\ncatX$ can be written as a sequence $s=(x_m\xrightarrow{s_{m,n}}x_n)_{m\leq n\in\ncatN}$ of arrows in \(\ncatX\) where, for all \(m\leq n\leq l\) in \(\ncatN\), \(s_{m,m}=1_{x_m}\) and \(s_{m,l}=s_{n,l}\cdot s_{m,n}\).

\begin{definition}
  Let \(s=(x_m\xrightarrow{s_{m,n}}x_n)_{m\leq n\in\ncatN}\) be a sequence in a $\mcatV$-normed category $\ncatX$. An object $x$ together with a cocone $\gamma=(x_n\xrightarrow{\gamma_n}x)_{n\in \ncatN}$ is a \emph{normed colimit} of $s$ in $\ncatX$ if
  \begin{enumerate}[wide, labelindent=5pt, leftmargin=*]
  \item[(C1)] $\gamma \colon s\to \Delta x$ is a colimit cocone in the ordinary category $\ncatX$, and
  \item[(C2)] for all objects $y$ in $\ncatX$, the canonical $\Set$-bijections
    \begin{displaymath}
      \begin{tikzcd}
        \mathrm{Nat}(s_{|N},\Delta y) %
        \ar{r}{\kappa_N} %
        & \ncatX(x,y), %
        && (f\cdot\gamma_n)_{n\geq N} %
        & f\ar[mapsto]{l}[swap]{\kappa_N^{-1}} %
        & (N\in\mathbb{N})
      \end{tikzcd}
    \end{displaymath}
    form a colimit cocone in $\SetV$, where $s_{|N}$ is the restriction of $s$ to $\uparrow\! N=\{N, N+1, \dots\}$ and $\mathrm{Nat}(s_{|N},\Delta y)=[\uparrow\! N,\ncatX](s_{|N},\Delta y)\;(N\in\ncatN)$ is considered as a sequence in $\Setss{\mcatV}$, with all connecting maps given by restriction. Also note that a colimit $x$ of $s$ in the ordinary category $\ncatX$ is also a colimit of every restricted sequence $s_{|N}$.
  \end{enumerate}
\end{definition}

We also recall from \citep{CHT25} the following characterisation of normed colimits.

\begin{corollary}
  An object $x$ with a cocone $\gamma \colon s\to \Delta x$ is a normed colimit of a sequence $s$ in a $\mcatV$-normed category $\ncatX$ if and only if $\gamma$ is a colimit cocone in the ordinary category $\ncatX$ such that
  \begin{enumerate}[wide, labelindent=5pt, leftmargin=*]
  \item[\mylabel{d:normed1}{\normalfont{(C2a)}}] $\kk\leq\bigvee_{N\in\ncatN}\bigwedge_{n\geq N}|\gamma_n|$,
  \item[\mylabel{d:normed2}{\normalfont{(C2b)}}] $|f|\geq \bigvee_{N\in\ncatN} \bigwedge_{n\geq N}|f\cdot\gamma_n|$, for every morphism $f \colon x\to y$ in $\ncatX$.
  \end{enumerate}
\end{corollary}

\begin{definition}
  Let \(\mcatV\) be a quantale and \(\ncatX\) be a \(\mcatV\)-normed category. A sequence \(s=(x_m\xrightarrow{s_{m,n}}x_n)_{m\leq n\in\ncatN}\) is \emph{(forward) Cauchy} whenever \(\kk\leq \bigvee_{N\in\ncatN}\bigwedge_{N\leq m\leq n}|s_{m,n}|\). The \(\mcatV\)-normed category \(\ncatX\) is \emph{Cauchy cocomplete} whenever every Cauchy sequence has a normed colimit in \(\ncatX\).
\end{definition}

\subsection{Examples of Cauchy cocomplete normed categories}
\label{sec:exampl-cauchy-cocomp}

Several \(\mcatV\)-normed categories were shown to be Cauchy cocomplete in \citep{CHT25}, in particular, the presheaf categories \([\ncatX,\nSetV]\) are Cauchy cocomplete \emph{under certain conditions on \(\mcatV\)}. Our first result shows that, at least for \(\ncatX=\ncatE\), this result holds for all quantales \(\mcatV\).

\begin{theorem}\label{d:thm:3}
  For every quantale \(\mcatV\), the \(\mcatV\)-normed category \(\nSetV\) is Cauchy cocomplete.
\end{theorem}
\begin{proof}
  The proof proceeds as in the second part of the proof of \citep[Theorem~7.1]{CHT25}. Consider a Cauchy sequence \(s=(A_m\xrightarrow{s_{m,n}}A_n)_{m\leq n\in\ncatN}\) in \(\nSetV\), hence, \(\kk\leq \bigvee_{N\in\ncatN}\bigwedge_{n\geq m\geq N}|s_{m,n}|\). Take the colimit \((\gamma_n \colon A_n\to A)_{n\in\ncatN}\) in \(\Set\) and structure \(A\) by
  \begin{displaymath}
    |a|=\bigwedge_{N\in\ncatN}\Phi_a(N),\qquad\text{where }
    \Phi_a(N)=\bigvee_{n\geq N}\bigvee_{a'\in\gamma_n^{-1}(a)}|a'|,
  \end{displaymath}
  for all \(a\in A\). Note that, since \(\Phi_a\) is decreasing in \(N\), one has \(a=\bigwedge_{N\geq K}\Phi_a(N)\) for every \(K\in\ncatN\). Also note that, for every \(n\geq K\) in \(\ncatN\) and every \(a\in A_K\), one has \(|s_{K,n}|\otimes |a|\leq |s_{K,n}(a)|\). To see~\ref{d:normed1}, we show first that, for all \(K\in\ncatN\),
  \begin{displaymath}
    \bigwedge_{n\geq K}|s_{K,n}|\leq|\gamma_K|.
  \end{displaymath}
  Take \(a\in A_K\), then
  \begin{align*}
    |a|\otimes \bigwedge_{n\geq K}|s_{K,n}|
    &\leq |a|\otimes\bigwedge_{m\geq K}\bigvee_{n\geq m}|s_{K,n}|\\
    &\leq \bigwedge_{m\geq K}\bigvee_{n\geq m}
      \underbrace{|s_{K,n}|\otimes |a|}_{\leq |s_{K,n}(a)|}\\
    &\leq \bigwedge_{m\geq K}\Phi_{\gamma_K(a)}(m)=|\gamma_K(a)|.
  \end{align*}
  Therefore
  \begin{displaymath}
    \bigwedge_{n\geq K}|s_{K,n}|
    \leq\bigwedge_{a\in A_K}[|a|,|\gamma_K(a)|]=|\gamma_K|.
  \end{displaymath}
  Finally,
  \begin{displaymath}
    \kk\leq\bigvee_{N\in\ncatN}\bigwedge_{K\geq N}\bigwedge_{n\geq K}|s_{K,n}|
    \leq \bigvee_{N\in\ncatN}\bigwedge_{K\geq N}|\gamma_K|,
  \end{displaymath}
  which proves \ref{d:normed1}. Regarding \ref{d:normed2}, let \(\varphi \colon A\to B\) in \(\nSetV\), \(a\in A\) and \(K\in\ncatN\). Then
  \begin{align*}
    \Big(\bigwedge_{n\geq K}|\varphi\cdot\gamma_n|\Big)
    \otimes\Big(\bigwedge_{N\geq K}\Phi_a(N)\Big)
    &\leq \bigwedge_{N\geq K}
      \Big(\Big(\bigwedge_{n\geq K}|\varphi\cdot\gamma_n|\Big)
      \otimes
      \Big(\bigvee_{m\geq N}\bigvee_{a'\in\gamma_m^{-1}(a)}|a'|\Big)\Big)\\
    &= \bigwedge_{N\geq K}\bigvee_{m\geq N}\bigvee_{a'\in\gamma_m^{-1}(a)}
      \Big(\Big(\bigwedge_{n\geq K}|\varphi\cdot\gamma_n|\Big)
      \otimes |a'|\Big)\\
    &\leq \bigwedge_{N\geq K}\bigvee_{m\geq N}\bigvee_{a'\in\gamma_m^{-1}(a)}
      \bigwedge_{n\geq K}(|\varphi\cdot\gamma_n|\otimes |a'|)\\
    &\leq \bigwedge_{N\geq K}\bigvee_{m\geq N}\bigvee_{a'\in\gamma_m^{-1}(a)}
      (|\varphi\cdot\gamma_m|\otimes |a'|)\\
    &\leq \bigwedge_{N\geq K}\bigvee_{m\geq N}\bigvee_{a'\in\gamma_m^{-1}(a)}
      |\varphi\cdot\gamma_m(a')|\leq|\varphi(a)|.
  \end{align*}
  Therefore
  \begin{displaymath}
    \bigwedge_{n\geq K}|\varphi\cdot\gamma_n|
    \leq \bigwedge_{a\in A}[|a|,\varphi(a)|]=|\varphi|,
  \end{displaymath}
  which proves \ref{d:normed2}.
\end{proof}

\begin{remark}
  We note that \(\nSetV\) is also Lawvere complete, by Remark~\ref{d:rem:2}.
\end{remark}

Let \(\DSets{\mcatV}\) be the normed category of ``\(\mcatV\)-distance sets'': \(\DSets{\mcatV}\) has as objects sets \(X\) equipped with an arbitrary (distance) function \(X(-,-)\colon X\times X\to\mcatV\), and as morphisms maps \(\varphi\colon X\to Y\) with norm
\begin{displaymath}
  |\varphi|=\bigwedge_{x,x' \in X}[X(x,x'),Y(\varphi x,\varphi x)].
\end{displaymath}
We note that there is a canonical norm-preserving functor
\begin{displaymath}
  D \colon\DSets{\mcatV}\longrightarrow\nSetV
\end{displaymath}
sending \(X\) to \(X\times X\).

\begin{proposition}\label{d:prop:5}
  For every quantale \(\mcatV\), the normed category \(\DSets{\mcatV}\) is Cauchy cocomplete.
\end{proposition}
\begin{proof}
  Let \(s=(X_m\xrightarrow{s_{m,n}}X_n)_{m\leq n}\) be a Cauchy sequence in \(\Cats{\mcatV}\). Then \(Ds=(X_m\times X_m\xrightarrow{s_{m,n}\times s_{m,n}}X_n\times X_n)_{m\leq n}\) is a Cauchy sequence in \(\nSetV\) and, by Theorem~\ref{d:thm:3}, has a \(\mcatV\)-normed colimit in \(\nSetV\). Since directed colimits commute with finite products in \(\Set\), this colimit must have the form
  \begin{displaymath}
    (X_n\times X_n\xrightarrow{\;\gamma_n\times\gamma_n\;}X\times X)_{n\in\mathbb{N}},
  \end{displaymath}
  where \((\gamma_n \colon X_n\to X)_{n\in \mathbb{N}}\) is a colimit in \(\Set\). Then Conditions~\ref{d:normed1} and \ref{d:normed2} hold for \(\gamma\) since they hold for \(D\gamma\).
\end{proof}

In the sequel we will consider two quantale structures on the lattice \(\mcatV\); besides the tensor \(\otimes\) with unit element \(\kk\) we have also the tensor \(\odot\) with unit element \(e\). For distincion, we write \(\mcatV_{\otimes}\) respectively \(\mcatV_{\odot}\) for the corresponding quantales. Then \(\VLipDot\) denotes the \(\mcatV_{\odot}\)-normed category of \(\mcatV_{\otimes}\)-categories and arbitrary maps, \(\mcatV_{\odot}\)-normed as in \(\DSets{\mcatV_{\odot}}\). We simply write \(\VLip\) if both quantale structures coincide. We recall that, for \(u,v\in\mcatV\), \(u\) is \emph{totally below} \(v\), in symbols: \(u\lll v\), whenever, for all \(S\subseteq\mcatV\), \(v\leq S\) implies \(u\leq s\) for some \(s\in S\). Moreover, \(v\) is \emph{approximated from totally below} whenever \(v=\bigvee\{u\in\mcatV\mid u\lll v\}\).

\begin{theorem}
  \label{d:thm:4}
  Assume that the unit element \(e\) of \(\mcatV_{\odot}\) is approximated from totally below. Then \(\VLipDot\) is Cauchy cocomplete.
\end{theorem}
\begin{proof}
  We show that \(\VLipDot\), seen as a full subcategory of \(\DSets{\mcatV_{\odot}}\), is closed under the formation of normed colimits of Cauchy sequences in \(\VLipDot\). To this end, let \(s=(X_m\xrightarrow{s_{m,n}}X_{n})\) be a Cauchy sequence in \(\VLipDot\), hence \(e\leq \bigvee_{N\in\mathbb{N}}\bigwedge_{n\geq m\in\mathbb{N}}|s_{m,n}|\). Let \((\gamma_n \colon X_n\to X)_{n\in\mathbb{N}}\) be a normed colimit of \(s\) in \(\DSets{\mcatV_{\odot}}\), in particular, for all \(x,y\in X\) and any \(K\in\mathbb{N}\),
  \begin{displaymath}
    X(x,y)=\bigwedge_{N\geq K}\Phi_{x,y}(N),
    \quad
    \Phi_{x,y}(N)=\bigvee_{n\geq N}\bigvee_{x'\in\gamma_n^{-1}(x)}
    \bigvee_{y'\in\gamma_n^{-1}(y)}X_n(x',y').
  \end{displaymath}
  Clearly, \(e\leq X(x,x)\) for all \(x\in X\). Let now \(x,y,z\in X\) and let \(\varepsilon,\eta\lll e\). Since \(s\) is Cauchy, we can find \(K\in \mathbb{N}\) so that, for all \(n\geq m\geq K\), \(\max\{\varepsilon,\eta\}\leq |s_{m,n}|\). Firstly,
  \begin{align*}
    (\varepsilon\odot X(x,y))\otimes(\eta\odot X(y,z))
    &= \Big(\varepsilon\odot
      \bigwedge_{N\geq K}\Phi_{x,y}(N)\Big)
      \otimes\Big(\eta\odot
      \bigwedge_{M\geq K}\Phi_{y,z}(M)\Big)\\
    &\leq \bigwedge_{N\geq K}
      \big(\varepsilon\odot\Phi_{x,y}(N)\big)
      \otimes
      \big(\eta\odot\Phi_{y,z}(N)\big)\\
    &=\bigwedge_{N\geq K}
      \bigvee_{\substack{n\geq N,\\ m\geq N}}
    \bigvee_{\substack{\gamma_n(x')=x,\\\gamma_n(y')=y}}
    \bigvee_{\substack{\gamma_m(y'')=y,\\\gamma_m(z')=z}}
    \big(\varepsilon\odot X(x',y')\big)
    \otimes
    \big(\eta\odot X(y'',z')\big).
  \end{align*}
  Let now \(N\geq K\), \(m,n\geq N\) and \(x',y',y'',z'\in X\) with \(\gamma_n(x')=x\), \(\gamma_n(y')=y\), \(\gamma_m(y')=y\) and \(\gamma_m(z')=z\). Since \(\gamma_n(y')=\gamma_m(y'')\), there is some \(l\in \mathbb{N}\) with \(n\leq l\), \(m\leq l\) and \(s_{n,l}(y')=s_{m,l}(y'')\). Then we get
  \begin{align*}
    \big(\varepsilon\odot X(x',y')\big)
    \otimes
    \big(\eta\odot X(y'',z')\big)
    &\leq \big(|s_{n,l}|\odot X(x',y')\big)
      \otimes
      \big(|s_{m,l}|\odot X(y'',z')\big)\\
    &\leq X_l(s_{n,l}(x'),s_{n,l}(y'))
      \otimes
      X_l(s_{m,l}(y''),s_{m,l}(z'))\\
    &\leq X_l(s_{n,l}(x'),s_{m,l}(z'))
      \leq\Phi_{x,z}(N).
  \end{align*}
  We conclude that
  \begin{displaymath}
    X(x,y)\otimes X(y,z)
    =\bigvee_{\varepsilon\lll e}(\varepsilon\odot X(x,y))
      \otimes
      \bigvee_{\eta\lll e}(\eta\odot X(y,z))\\
    \leq X(x,z).\qedhere
  \end{displaymath}
\end{proof}

\begin{remark}
  A normed colimit in \(\VLipDot\) of a Cauchy sequence of symmetric \(\mcatV_{\otimes}\)-categories is symmetric.
\end{remark}

As a corollary, choosing \(\odot=\otimes\) and \(e=\kk\), we obtain \cite[Theorem~7.1]{CHT25}.

\begin{corollary}
  Assume that the unit element \(\kk\) of \(\mcatV\) is approximated from totally below. Then \(\VLip\) is Cauchy cocomplete.
\end{corollary}

One obtains another interesting instance of Theorem~\ref{d:thm:4} by considering the lattice \([0,\infty]\) with \(\otimes=+\) and neutral element \(\kk=0\), and \(\odot=\cdot\) multiplication with \(e=1\). Here \(\alpha\cdot\infty=\infty\) for every \(\alpha\in [0,\infty]\) since the tensor preserves the bottom element. The internal hom is given by \([\alpha,\beta]=\frac{\beta}{\alpha}\) for all \(\alpha,\beta\in [0,\infty]\); here
%\begin{displaymath}
\(  \frac{0}{0}=0,\quad
  \frac{\alpha}{0}=\infty\;(\alpha>0),\quad
  \frac{\alpha}{\infty}=\frac{\infty}{\infty}=0.\)
%\end{displaymath}

\begin{corollary}
  The category of Lawvere metric spaces and arbitrary maps \(\varphi \colon X\to Y\) between them, normed by
  \begin{displaymath}
    |\varphi|=\sup_{x,x'\in X}\frac{Y(\varphi x,\varphi x')}{X(x,x')},
  \end{displaymath}
  is Cauchy cocomplete.
\end{corollary}

\begin{corollary}
  The normed category \(\Met_{\infty}\) of Lawvere metric spaces and arbitrary maps \(\varphi \colon X\to Y\) between them, normed by
  \begin{displaymath}
    |\varphi|=\max\{0,\log\Big(\sup_{x,x'\in X}\frac{Y(\varphi x,\varphi x')}{X(x,x')}\Big)\},
  \end{displaymath}
  is Cauchy cocomplete.
\end{corollary}

% \bibliographystyle{hapalike}
% \bibliography{bibliography_dirk}

\end{document}